\documentclass[12pt]{amsart}

\usepackage{amssymb}

\usepackage{enumitem}

\usepackage{graphicx}

\makeatletter
\@namedef{subjclassname@2010}{%
  \textup{2010} Mathematics Subject Classification}
\makeatother

\usepackage[T1]{fontenc}
\newtheorem{theorem}{Theorem}[section]

\newtheorem{lemma}[theorem]{Lemma}
\newtheorem{proposition}[theorem]{Proposition}

\theoremstyle{definition}

\numberwithin{equation}{section}

\begin{document}

%%%%% To ease editing, for IMPAN journals add:

\baselineskip=17pt

%%%%%%%%%%%%%%%%

\title[A stability problem for some complete and minimal Gabor systems in $L^2(\mathbb{R})$\ \ \ ]{A stability problem for some complete and minimal Gabor systems in $L^2(\mathbb{R})$}

\author[Y. Omari]{Youssef Omari}
\address{Laboratory of Mathematical Analysis and Applications (LAMA), CeReMAR, Faculty of Sciences, Mohammed V University in Rabat, 4 Av. Ibn Battouta, Morocco.}
\email{omariysf@gmail.com}

\date{}

\begin{abstract}
A Gabor system in $L^2(\mathbb{R})$, generated by a window $g\in L^2(\mathbb{R})$ and associated with a sequence of times and frequencies $\Gamma\subset\mathbb{R}^2$, is a set formed by translations in time and modulations of $g$. In this paper we consider the case when $g$ is the Gaussian function and $\Gamma$ is a sequence whose associated Gabor system $\mathcal{G}_\Gamma$ is complete and minimal in $L^2(\mathbb{R})$. {We consider two main cases: that of the lattice without one point and that of the sequence constructed by Ascensi, Lyubarskii and Seip lying on the union of the coordinate axes of the time-frequency space. We study the stability problem for these two systems}. {More precisely, we  describe the perturbations of $\Gamma$ such that the associated Gabor systems remain to be complete and minimal}. Our method of proof is based essentially on  estimates of some infinite products.
\end{abstract}



\maketitle

\section{Introduction}
Gabor analysis is a method, introduced in \cite{gabor1946theory}, to express arbitrary square integrable functions using  a countable set with minimal support in the time-frequency plane. Gabor systems are sets  formed by translations in time and frequency of a single function. The signals are represented as  series over the Gabor systems. We refer to \cite{grochenig2013foundations,grochenig2016completeness,lyubarskii1999convergence} for more details on this theory.

A \textit{Gabor system} generated by a function $g\in L^2(\mathbb{R})$ and associated with a sequence $\Lambda$ of distinct points in $\mathbb{R}^2$ is given by
\begin{eqnarray*}
\mathcal{G}_{\Lambda} := \Big\{\rho_{(x,y)}g(t) := e^{2i\pi yt}g(t-x)\ :\ (x,y)\in\Lambda\Big\}.
\end{eqnarray*}
Throughout this paper $g$ is the Gaussian window $e^{-\pi t^2}$. The element $\rho_{(x,y)}g$ is the time-frequency shift of $g$, with respect to a point $(x,y)$, in the phase space. By definition,
\[\rho_{(x,y)}g(t) = M_yT_xg(t),\quad x,y,t\in\mathbb{R},\]
where $T_x$ and $M_y$ are two {unitary} operators in $L^2(\mathbb{R})$ given by
\[T_x : f\ \longmapsto\ T_xf(t):=f(t-x),\quad M_y : f \longmapsto\ M_y f(t):=e^{2i\pi yt}f(t).\]

It is well-known that $\mathcal{G}_\Lambda$ cannot be a Riesz basis in $ L^2(\mathbb{R})$ (see \cite{borichev2017geometric,brekke1993density,seip1992density,seip1992densityy}). However, there are numerous complete and minimal Gabor systems. A typical example is the lattice without one point, $\Gamma_1:=\mathbb{Z}\times\mathbb{Z}\setminus\{(0,0)\}$, (see \cite[Theorem 5.9]{zhu2012analysis}). More generally, for every $\nu\in(0,1]$ the system $\mathcal{G}_{\Gamma_\nu}$ is complete and minimal (see \cite[Theorem 2]{lyubarskii1992frames}), where $\Gamma_\nu$ is given by
\begin{equation}\nonumber
\Gamma_\nu:=\Big(\mathbb{Z}\times\left(\mathbb{Z}\setminus\{0\}\right)\Big)\cup\Big(\mathbb{Z}_{-}\times\{0\}\Big)\cup\Big\{(m+\nu,0)\Big\}_{m\geq 0}.
\end{equation}
It is easy to see that the uniform Beurling density of these sequences is 1 and they are separated. As a consequence, $\mathcal{G}_{\Gamma_\nu}$ is neither a Gabor frame nor a Riesz sequence in $ L^2(\mathbb{R})$ (see \cite{seip1992density,seip1992densityy}). Another interesting  example of a complete and minimal Gabor system in $ L^2(\mathbb{R})$ was provided by Ascensi, Lyubarskii and Seip in \cite{ascensi2009phase}. This system is associated with a non separated sequence {lying} on the coordinate axes of the plane $\mathbb{R}^2$.  It turns out that the lower Beurling density of this sequence is zero, hence the associated Gabor system is {very far from being a Gabor frame or a Riesz sequence in $ L^2(\mathbb{R})$}.  The sequence is given by
\begin{equation}\nonumber
\Gamma := \left\{(\pm\sqrt{2n},0),\ (0,\pm\sqrt{2n})\ :\quad n\geq 1\right\}\cup\{(\pm 1,0)\}.
\end{equation}

Later on, Belov proved in \cite{belov2015uniqueness}  that if a Gabor system $\mathcal{G}_\Lambda$ is complete and minimal in $ L^2(\mathbb{R})$, then its unique biorthogonal system, noted $\big\{h_{(x,y)}\ :\ (x,y)\in\Lambda \big\}$, is also complete and minimal in $ L^2(\mathbb{R})$. This means that every complete and minimal Gabor system constitutes a \textit{Markushevich basis} or simply an {\it M-basis} in $ L^2(\mathbb{R})$. Thus, every function in $ L^2(\mathbb{R})$ (or, equivalently, by the Bargmann transform in the Fock space) is uniquely determined by its Fourier coefficients with respect to such Gabor system. We write
\begin{eqnarray}\nonumber
f \sim \sum_{(x,y)\in\Lambda} \langle f,h_{(x,y)}\rangle\ \rho_{(x,y)}g,\quad f\in { L}^{2}(\mathbb{R}).
\end{eqnarray}
For the convergence of such series we refer to the works by Lyubarskii and Seip in \cite{lyubarskii1999convergence} and by Dumont and Kellay in \cite{dumont2012convergence}. 
Throughout this paper, we identify $\mathbb{R}^2$ with the complex plane $\mathbb{C}$ and we write
{\small\begin{equation}\label{lambda_nu}
\Gamma_\nu=\left\{\gamma_{m,n}:=m+in\ :\  (m,n)\in\mathcal{I}_\nu\right\}  =\{\gamma\}_{\gamma\in\Gamma_\nu},
\end{equation}}
where  $\mathcal{I}_\nu=\Big(\mathbb{Z}\times\big(\mathbb{Z}\setminus\{0\}\big)\Big)\cup\Big(\mathbb{Z}_{-}\times\{0\}\Big)\cup\Big\{(m'+\nu,0)\Big\}_{m'\geq 0}$ and 
\begin{equation}\label{ze}
\Gamma=\left\{\pm\sqrt{2n},\ \pm i\sqrt{2n}\ :\ n\geq 1\right\}\cup\{\pm 1\}=\{\gamma\}_{\gamma\in\Gamma}.
\end{equation}

A natural question in the context of the time-frequency analysis is the study of complete and minimal Gabor systems. The question we deal with in the present paper is, {for a given $\Sigma$ where $\Sigma=\Gamma$ or $\Gamma_\nu$,} to characterize  the multiplicatively perturbed systems $\mathcal{G}_\Lambda$ that remain to be complete and minimal; here
$$\Lambda:=\left\{\lambda_\sigma:=\sigma e^{\delta_\sigma}e^{i\theta_\sigma}\ :\ \sigma\in\Sigma\right\}.$$
For different problems concerning spectral synthesis in the Fock type spaces see  \cite{baranov2015spectral,baranov2016completness} and the references therein.
\subsection{Finite strip perturbation of $\Gamma_\nu$}
To simplify the statements, we begin by dealing with  perturbations on a finite strip, the general case will be presented in the next subsection. Our first result in such context is the following 
\begin{theorem}\label{corA}
Let $N$ be a positive integer and let $\Lambda_N:=\{\lambda_{m,n}\ :\ m\in\mathbb{Z},\ 1\leq n\leq N\}$
be a sequence of complex numbers. Let 
$\Lambda:=\{\gamma_{m,n}\in\Gamma_\nu\ :\ (m,n)\in\mathbb{Z}\times\mathbb{Z}\setminus\{1\leq n\leq N\} \}\cup\Lambda_N.$
If  
\begin{eqnarray}\label{cd}
\left|\gamma_{m,n}-\lambda_{m,n}\right|\leq \frac{\delta}{N},\quad (m,n)\in\mathbb{Z}\times\{1,\cdots,N\},
\end{eqnarray}
for some $\delta<\frac{\min\{\nu,1-\nu\}}{2}$, then  the Gabor system $\mathcal{G}_\Lambda$ associated to $\Lambda$ is complete and minimal in $ L^2(\mathbb{R})$.
\end{theorem}

When $\nu \leq 1/2$, the condition $\delta<\frac{\min\{\nu,1-\nu\}}{2}$ in Theorem \ref{corA} is optimal. Indeed, for fixed $N\geq 1$, we denote
     $$\Lambda_{\beta,N} :=\left\{m+\mathrm{sign}(m)\frac{\beta}{N}+i n\ :\ m\in\mathbb{Z},\ 1\leq n\leq N\right\},$$
where $\beta$ is a real number and $\mathrm{sign}$ is defined by
$$\mathrm{sign}(m):=\left\{\begin{array}{cc}
+1 & \mbox{if}\ m\geq 1,\\
0  & \mbox{if}\ m=0,\\
-1 & \mbox{if}\ m\leq -1.
\end{array}
\right.
$$
Next consider the sequences
     $$\Lambda_{\beta,\nu}:=\Big\{\gamma_{m,n}\in\Gamma_\nu\ :\ (m,n)\in\mathbb{Z}\times\left(\mathbb{Z}\setminus\{1\leq n\leq N\}\right) \Big\}\cup\Lambda_{\beta,N}.$$
We have $ \left|\gamma_{m,n}-\lambda_{m,n}\right|=\frac{|\beta|}{N}$. If $\beta=-\nu$, then  $\delta=\frac{\nu}{2}$. In this case, we prove that $\mathcal{G}_{\Lambda_{\beta,\nu}}$ is complete but not minimal in ${ L}^2(\mathbb{R})$. More precisely, 
we have the following result.

\begin{theorem}\label{thm2}
Let $0<\nu\leq 1,$ $\beta\in \mathbb{R}$. Then we have 
\begin{enumerate}[label=$\arabic*)$,leftmargin=* ,parsep=0cm,itemsep=0cm,topsep=0cm]
\item If $\beta \leq -\frac{\nu}{2}$, then  $\mathcal{G}_{\Lambda_{\beta,\nu}}$ is complete and not minimal in $ L^2(\mathbb{R})$.
\item If $-\frac{\nu}{2} < \beta\leq \frac{1-\nu}{2}$, then $\mathcal{G}_{\Lambda_{\beta,\nu}}$ is complete and minimal in $ L^2(\mathbb{R}).$
\item If $\beta > \frac{1-\nu}{2}$, then $\mathcal{G}_{\Lambda_{\beta,\nu}}$ is not complete in $ L^2(\mathbb{R})$.
\end{enumerate}
\end{theorem}
For a real number sequence $(\delta_{k,j})_{(k,j)\in\mathbb{Z}\times\{1,\cdots,N\}}$ we use the notations :
\[
\delta(\Lambda) := \underset{m\rightarrow\infty}{\limsup}\ \frac{1}{\log m}\underset{1\leq j\leq N}{\underset{-m\leq k\leq m}{\sum}}\ \delta_{k,j},\quad\mbox{and}\quad \hat{\delta}(\Lambda) := \underset{m\rightarrow\infty}{\liminf}\ \frac{1}{\log m}\underset{1\leq j\leq N}{\underset{-m\leq k\leq m}{\sum}}\ \delta_{k,j}.
\]
The following result is a generalization of Theorem \ref{corA}. 
\begin{theorem}\label{thm1}
Let $N$ be a positive integer and let $\Lambda_N:=\{\lambda_{m,n}\ :\ m\in\mathbb{Z},\ 1\leq n\leq N\}$ be a sequence of complex numbers. We write $\lambda_{m,n}=\gamma_{m,n}e^{\delta_{m,n}}e^{i\theta_{m,n}}$, for every $m\in\mathbb{Z}$ and $1\leq n\leq N$, where $\delta_{m,n},\theta_{m,n}\in\mathbb{R}.$ Let 
$$\Lambda:=\Big\{\gamma_{m,n}\in\Gamma_\nu\ :\ (m,n)\in\mathbb{Z}\times\big(\mathbb{Z}\setminus\{1\leq n\leq N\}\big) \Big\}\cup\Lambda_N.$$
Suppose that 
\begin{enumerate}[label=$\alph*)$,leftmargin=* ,parsep=0cm,itemsep=0cm,topsep=0cm]
\item\label{thm10} $\Lambda$ is separated,
\item\label{thm11} The sequences $\left(m\delta_{m,n}\right)_{(m,n)\in\mathbb{Z}\times\{1\leq n\leq N\}}$ and $\left(m\theta_{m,n}\right)_{(m,n)\in\mathbb{Z}\times\{1\leq n\leq N\}}$ are bounded,
\item\label{thm12} \begin{eqnarray}
 \nu-1\ <\ \hat{\delta}(\Lambda)\ \leq \ \delta(\Lambda)\ <\ \nu.
  \label{delta}
\end{eqnarray}
\end{enumerate}
Then $\mathcal{G}_{\Lambda}$ is complete and minimal in $ L^2(\mathbb{R})$.
\end{theorem}

\subsection{{Unrestricted} perturbation of $\Gamma_\nu,\ 0< \nu\leq 1$}
In this subsection we are looking for perturbations of the whole sequence $\Gamma_\nu$. Our  results are similar to those of the previous subsection. We first establish  the following fact
\begin{theorem}\label{corAfull}
Let  $\Lambda:=\{\lambda_{m,n}\}_{m,n\in\mathbb{Z}}$ 
be a sequence of complex numbers such that
\begin{eqnarray}\label{cdfull}
\left|\gamma_{m,n}-\lambda_{m,n}\right|\leq \frac{\delta}{|\gamma_{m,n}|},
\end{eqnarray}
for some $\delta<\min\{\nu,1-\nu\}$. Then   $\mathcal{G}_\Lambda$ is complete and minimal in $ L^2(\mathbb{R})$.
\end{theorem}

The condition $\delta<\min\{\nu,1-\nu\}$ in Theorem \ref{corAfull} is optimal whenever $\nu\leq 1/2$. Indeed, for a given real parameter $\beta$ we consider the following sequence
\begin{eqnarray}
\widetilde{\Lambda}_{\beta,\nu} :=\left\{\lambda_\gamma:=\gamma e^{\frac{\beta}{|\gamma|^2}}\ :\ \gamma\in\Gamma_\nu \right\}.
\end{eqnarray}
We simply remark that 
\[
\underset{|\gamma|\rightarrow\infty}{\lim}|\gamma||\gamma-\lambda_\gamma|=\underset{|\gamma|\rightarrow\infty}{\lim}|\gamma|^2\left|1-e^{\frac{\beta}{|\gamma|^2}}\right| = |\beta|.
\]
However, if $\nu\leq 1/2$ and $\beta=-\nu,$ we have  $\delta=|\beta|=\nu$ and the system $\mathcal{G}_{\widetilde{\Lambda}_{\beta,\nu}}$ is complete and not minimal in $ L^2(\mathbb{R})$. More precisely, we prove the following result
\begin{theorem}\label{thm2full}
Given a real parameter $\beta$, the following hold :
\begin{enumerate}[label=$\arabic*)$,leftmargin=* ,parsep=0cm,itemsep=0cm,topsep=0cm]
\item If $\beta \leq -\nu$, then  $\mathcal{G}_{\widetilde{\Lambda}_{\beta,\nu}}$ is complete and not minimal in $ L^2(\mathbb{R})$.
\item If $-\nu < \beta\leq 1-\nu$, then $\mathcal{G}_{\widetilde{\Lambda}_{\beta,\nu}}$ is complete and minimal in $ L^2(\mathbb{R}).$
\item If $\beta > 1-\nu$, then $\mathcal{G}_{\widetilde{\Lambda}_{\beta,\nu}}$ is not complete in $ L^2(\mathbb{R})$.
\end{enumerate}
\end{theorem}

Now, if $(\delta_\gamma)_{\gamma\in\Gamma_\nu}$ is a sequence of real numbers, we put 
\[
\delta(\Lambda):= \underset{R\rightarrow\infty}{\limsup}\  \frac{1}{\log R}{\underset{|\gamma|\leq R}{\sum}}\ \delta_{\gamma},\quad\mbox{and}\quad \hat{\delta}(\Lambda):= \underset{R\rightarrow\infty}{\liminf}\ \frac{1}{\log R}{\underset{|\gamma|\leq R}{\sum}}\ \delta_{\gamma}.
\]
Theorem \ref{corAfull} will be obtained as a consequence of the following more general result. 
\begin{theorem}\label{thm1full}
Let  $\Lambda:=\{\lambda_{\gamma}\}_{\gamma\in\Gamma_\nu}$ 
be a sequence of complex numbers. We write $\lambda_{\gamma}=\gamma e^{\delta_{\gamma}}e^{i\theta_{\gamma}}$, for every $\gamma\in\Gamma_\nu$, where $\delta_{\gamma},\theta_{\gamma}\in\mathbb{R}.$ Assume that 
\begin{enumerate}[label=$\roman*)$,leftmargin=* ,parsep=0cm,itemsep=0cm,topsep=0cm]
\item\label{thm10full} $\Lambda$ is separated,
\item\label{thm11full} The sequences $\left(\gamma^2\delta_{\gamma}\right)_{\gamma\in\Gamma_\nu}$ and $\left(\gamma^2\theta_{\gamma}\right)_{\gamma\in\Gamma_\nu}$ are bounded,
\item\label{thm12full} \begin{eqnarray}
 \nu-1\ <\ \hat{\delta}(\Lambda)\ \leq \ \delta(\Lambda)\ < \nu%\frac{1}{2}
 . \label{deltafull}
\end{eqnarray}
\end{enumerate}
Then $\mathcal{G}_{\Lambda}$ is complete and minimal in $ L^2(\mathbb{R})$.
\end{theorem}

\subsection{Perturbation of the Ascensi-Lyubarskii-Seip sequence}
 In order to simplify the statements, we present first our results by looking at the perturbation on a half-axis. Theorem \ref{corGaborALS} gives the  general case. 
\begin{theorem}\label{corAALS}
Let $\Lambda:=\{\lambda_n,\ -\sqrt{2n},\ \pm i\sqrt{2n}\ :\ n\geq 1\}\cup\{\pm 1\}$ be a sequence of complex numbers. If
\begin{eqnarray}\label{cdALS}
\left|\sqrt{2n}-\lambda_n\right|\leq \delta/\sqrt{2n},\quad n\geq 1,
\end{eqnarray}
for some $\delta<1/2$, then  $\mathcal{G}_\Lambda$ is complete and minimal in $ L^2(\mathbb{R})$.
\end{theorem}

The condition $\delta<1/2$ in the last theorem is optimal. Indeed, consider the sequence
  $$\Lambda_\beta :=\Big\{\sqrt{2n+4\beta},\ -\sqrt{2n},\\ \pm i\sqrt{2n}\ :\ n\geq 1\Big\}\cup\{\pm 1\},$$
where $\beta$ is a real number. If $\beta$ is negative we replace $\sqrt{2n+4\beta}$, in the definition of $\Lambda_\beta$, by $\sqrt{2n}$ for every $1\leq n<-2\beta$. We have $\underset{n\rightarrow\infty}{\lim}\ \sqrt{2n}\left|\sqrt{2n}-\sqrt{2n+4\beta}\right|=2|\beta|$. For $\beta=-\frac{1}{4}$, so $\delta=1/2$ in this case, we can prove that $\mathcal{G}_{\Lambda_\beta}$ is not minimal in ${ L}^2(\mathbb{R})$. More precisely, we have the following result.

\begin{theorem}\label{thm2ALS}
Let $\beta\in \mathbb{R}$ be given. Then we have :
\begin{enumerate}[label=$\arabic*)$,leftmargin=* ,parsep=0cm,itemsep=0cm,topsep=0cm]
\item If $\beta \leq -1/4$, then  $\mathcal{G}_{\Lambda_\beta}$ is complete and not minimal in $ L^2(\mathbb{R})$.
\item If $-1/4 < \beta\leq 1/4$, then $\mathcal{G}_{\Lambda_\beta}$ is complete and minimal in $ L^2(\mathbb{R}).$
\item If $\beta > 1/4$, then $\mathcal{G}_{\Lambda_\beta}$ is not complete in $ L^2(\mathbb{R})$.
\end{enumerate}
\end{theorem}
Theorem \ref{corAALS} is a consequence of a more general result :
\begin{theorem}\label{thm1ALS}
Let $\Lambda:=\left\{ \lambda_n,\ -\sqrt{2n},\ \pm i\sqrt{2n}\ :\ n\geq 1\right\}\cup\{\pm  1\}$ be a sequence of complex numbers such that $\lambda_{n}=\sqrt{2n}e^{\delta_n}e^{i\theta_n}$, for every $n\geq 1$. Suppose that
\begin{enumerate}[label=$\alph*)$,leftmargin=* ,parsep=0cm,itemsep=0cm,topsep=0cm]
\item\label{thm10ALS} There exists $c>0$ such that $\left|\lambda_{m}-\lambda_n\right|\geq c/\min\{\sqrt{m},\sqrt{n}\}$, for every $n,m\geq 1$.
\item\label{thm11ALS} The sequences $\left(n\delta_n\right)_{n\geq 1}\ \mbox{and}\ \left(n\theta_n\right)_{n\geq 1}$ are bounded.
\item\label{thm12ALS}{\small \begin{eqnarray}
\delta(\Lambda):= \underset{n\rightarrow\infty}{\limsup}\,\ \frac{1}{\log n}\left|\underset{k =1}{\overset{n}{\sum}}\ \delta_{k}\right|  < \frac{1}{4}. \label{deltaA}
\end{eqnarray}}
\end{enumerate}
Then $\mathcal{G}_{\Lambda}$ is complete and minimal in $ L^2(\mathbb{R})$.
\end{theorem}

Now, for perturbations of the whole sequence  let $\Lambda=\{\lambda_\gamma:=\gamma e^{\delta_\gamma}e^{i\theta_\gamma}\ :\ \gamma\in \Gamma\}$ be a sequence in the complex plane. We denote
\begin{eqnarray}
\Delta_n := \sum_{|\gamma|= \sqrt{2n}} \delta_\gamma = \delta_{\sqrt{2n}} + \delta_{-\sqrt{2n}} + \delta_{i\sqrt{2n}} + \delta_{-i\sqrt{2n}},\quad n\geq 1. \nonumber
\end{eqnarray}
\begin{theorem}\label{corGaborALS}
Let $\Gamma$ be the sequence given in \eqref{ze} and let  $\Lambda=\{\lambda_\gamma:=\gamma e^{\delta_\gamma}e^{i\theta_\gamma}\ :\ \gamma\in \Gamma\}$ be a sequence in $\mathbb{C}$. Assume that
\begin{enumerate}[label=$\alph*)$,leftmargin=* ,parsep=0cm,itemsep=0cm,topsep=0cm]
\item There exists $c>0$ such that $\left|\lambda_{\gamma}-\lambda_{\gamma'}\right|\geq c/\min\{|\gamma|,|\gamma'|\}$, for every $\gamma,\gamma'\in\Gamma$.
\item\label{corGab1ALS} The sequences $\left(\gamma^2\delta_\gamma\right)_{\gamma\in\Gamma}\ \mbox{and}\ \left(\gamma^2\theta_\gamma\right)_{\gamma\in\Gamma}$ are bounded,
\item\label{corGab2ALS}  $\delta(\Lambda):= \underset{n\rightarrow\infty}{\limsup}\,\ \frac{1}{\log n}\left|\underset{k=1 }{\overset{n}{\sum}}\ \Delta_{k}\right|  < \frac{1}{4}$.
\end{enumerate}
Then $\mathcal{G}_\Lambda$ is complete and minimal in $ L^2(\mathbb{R})$.
\end{theorem}

{\bf Remarks.}  Before passing to the proofs of our results, we first give some remarks.

\begin{enumerate}[label=$\arabic*.$,leftmargin=* ,parsep=0cm,itemsep=0cm,topsep=0cm]
\item For $\nu=1/2$, conditions \eqref{delta} and \eqref{deltafull} become  {\small $\underset{R\rightarrow\infty}{\limsup}\,\ \frac{1}{\log R}\left|\underset{|\gamma|\leq R }{\sum}\ \delta_{\gamma}\right|  < \frac{1}{2}$}.  From Theorems \ref{thm2} and \ref{thm2full} it is clear that this condition is optimal. Indeed, it is easy to check that $\delta(\Lambda_{\beta,\nu})= \delta(\widetilde{\Lambda}_{\beta,\nu}) = |\beta|$, and $\mathcal{G}_{{\Lambda}_{-1/2,\nu}}$ and $\mathcal{G}_{\widetilde{\Lambda}_{-1/2,\nu}}$ are not complete and minimal in $ L^2(\mathbb{R})$. On the other hand, Theorem \ref{thm2ALS} ensures that condition \ref{thm12ALS} in Theorems \ref{thm1ALS} and \ref{corGaborALS} is optimal. In fact, it is simple to see that $\delta(\Lambda_\beta)=|\beta|$ and $\mathcal{G}_{\Lambda_{-1/4}}$ is not complete and minimal in $ L^2(\mathbb{R})$.\\

\item Let $\left(\delta_n\right)$ be a sequence of real numbers. If for some positive integer $M\geq 1$, we have
\begin{eqnarray}\label{(b)}
\underset{n\geq 0}{\sup}\ \frac{n+1}{M}\left|\underset{k=n+1}{\overset{n+M}{\sum}}\delta_k\right| <\ A,
\end{eqnarray}
then {\small $\underset{n\rightarrow\infty}{\limsup}\ \frac{1}{\log n}\ \left|\underset{k=1}{\overset{n}{\sum}}\ \delta_k\right|  <A$}. Analogously, if 
{\small
\begin{eqnarray}\label{(bb)}
\underset{n\geq 0}{\sup}\ \frac{n+1}{M}\left|\underset{n+1\leq|\gamma|\leq n+M}{\sum}\delta_\gamma\right|<A,
\end{eqnarray}} for some integer $M\geq 1$,
then {\small $ \underset{R\rightarrow\infty}{\limsup}\  \frac{1}{\log R}\left|{\underset{|\gamma|\leq R}{\sum}}\ \delta_{\gamma}\right|<A$.} The converse is not true (see Lemma \ref{equiv}).\\
\item In the case $M=1$, condition \eqref{(b)} is similar to the Kadets and Ingham 1/4 theorem dealing with the  stability problem of Riesz bases of complex exponentials for the Hilbert space $ L^2(-\pi,\pi)$ (see, e.g. \cite{hruvsvcev1981unconditional,kadets1964exact,levin1996lectures,seip2004interpolation}). Also, for arbitrary integer $M\geq 1$, Theorems \ref{thm1} and \ref{thm1ALS} with \eqref{(b)} and Theorem \ref{thm1full} with \eqref{(bb)} give results similar to the Avdonin theorem (see, e.g. \cite{avdonin1974question,baranov2015sampling}).\\
\end{enumerate}

{ The plan of the paper is as follows. In the next section, we present a link between the completeness and minimality nature of $\mathcal{G}_\Gamma$ and a uniqueness problem in the Bargmann-Fock space. In section \ref{aaa}, we obtain upper and lower estimates of some modified Weierstrass infinite products. Section \ref{coro-proof} is devoted to proving our main results. In the last section, we discuss the optimality of our conditions and the relations between conditions \eqref{delta} and \eqref{(b)}. \\}

Throughout this paper, the notation $U(x)\lesssim V(x)$ means that there exists a constant $C>0$ such that $0\leq U(x)\leq CV(x)$ holds for every $x$ in the set in question. We write $U(x)\asymp V(x)$ if both $U(x)\lesssim V(x)$ and $V(x)\lesssim U(x)$ hold simultaneously. \\

\section{Transition to the Fock space}

Throughout this paper, we denote  $d\mu(z):=e^{-\pi|z|^2}dA(z)$, where $dA(z)$ is Lebesgue area measure in the complex plane $\mathbb{C}$. The Bargmann-Fock space $\mathcal{F}$ consists  of all entire functions $f$ satisfying
\begin{eqnarray}
\|f\|_\mathcal{F}^2\ :=\ \int_\mathbb{C}\left|f(z)\right|^2\ d\mu(z)\ <\ \infty. \nonumber
\end{eqnarray}
The space $\mathcal{F}$ is a Hilbert space with the reproducing kernel  $$\mathrm{k}_{z}(w)=e^{\pi \overline{z}w},\quad z,w\in\mathbb{C}.$$ It is well-known that the Bargmann transform given by
\begin{eqnarray}
\mathcal{B} f(z)\ :=\ 2^{1/4}e^{-\frac{\pi}{2}z^2}\int_\mathbb{R} f(x)e^{2\pi xz-\pi x^2}\,dx,\nonumber
\end{eqnarray}
maps $ L^2(\mathbb{R})$ isometrically onto $\mathcal{F}$, see \cite[Theorem 6.8]{zhu2012analysis}.  Moreover, the time–frequency shifts of the Gaussian are mapped to the normalized reproducing kernels $\Bbbk_{z}:=\mathrm{k}_{z}/\left\|\mathrm{k}_{z}\right\|$ of $\mathcal{F}$, namely
\begin{eqnarray}\label{Bargmann}
2^{1/4}\mathcal{B}\left(e^{-i\pi \textmd{Re}(z)\textmd{Im}(z)}\rho_z(e^{-\pi x^2})\right) = \Bbbk_{\overline{z}},\quad z\in\mathbb{C}.
\end{eqnarray}
For more informations on the Bargmann transform we refer the reader  to \cite{grochenig2013foundations,zhu2012analysis}.\\

We recall that a sequence $\Lambda\subset\mathbb{C}$ of distinct points is said to be \textit{a set of uniqueness} for $\mathcal{F}$ if the unique function in $\mathcal{F}$ vanishing on $\Lambda$ is the zero function. We say also that a countable set $\Lambda\subset\mathbb{C}$ is \textit{a zero set} for $\mathcal{F}$ whenever there exists a non identically zero function $f\in\mathcal{F}$  such that $\Lambda$ is exactly the zero set of $f$, counting the multiplicities. \\

Following \cite{ascensi2009phase}, a sequence $\Lambda$ is called \textit{a set of uniqueness of zero excess} for $\mathcal{F}$ if it is a set of uniqueness for $\mathcal{F}$ and when we remove any point of $\Lambda$, it becomes a zero set for $\mathcal{F}$.\\

The formula \eqref{Bargmann} and a standard duality argument give the following result 
\begin{lemma}\label{00}
A system $\mathcal{G}_\Lambda$ is complete and minimal in $ L^2(\mathbb{R})$ if and only if the system $\left\{\Bbbk_\lambda\right\}_{\lambda\in\Lambda}$ is complete and minimal in $\mathcal{F}$, that is, if and only if the sequence $\Lambda$ is a uniqueness set of zero excess for $\mathcal{F}$.
\end{lemma}

We refer to \cite{ascensi2009phase,belov2015uniqueness} for the proof of the above lemma. It follows that the sequences
\begin{equation}\label{zed}
\Gamma:=\left\{\pm\sqrt{2n},\ \pm i\sqrt{2n}\ :\ n\geq 1\right\} \cup\left\{\pm  1\right\} = \{\gamma\}_{\gamma\in\Gamma}
\end{equation}
and \begin{equation}\label{llambda_nu}
\Gamma_\nu=\Big\{\gamma_{m,n}:=m+in\ :\  (m,n)\in\mathbb{Z}\times\left(\mathbb{Z}\setminus\{0\}\right)\Big\}\cup\mathbb{Z}_{-}\cup\{m+\nu\}_{m\geq 0} =\{\gamma\}_{\gamma\in\Gamma_\nu}
\end{equation}
are uniqueness sets of zero excess for $\mathcal{F}$.

\section{Key Lemmas}\label{aaa}
In this section, we introduce certain modified Weierstrass products and prove some estimates of these functions which will play a crucial role in the proof of our results. We first {consider the sequences
$$\Gamma_{\nu}:=\Big\{\gamma_{m,n}:=m+in\ :\ (m,n)\in\mathbb{Z}\times\left(\mathbb{Z}\setminus\{0\}\right)\Big\}\cup\mathbb{Z}_{-}\cup\{m+\nu\}_{m\geq 0},$$
where $\nu$ is a real number in the interval $(0,1]$. Now, let $N$ be a positive integer and let $\Lambda_N:=\{\lambda_{m,n}\ :\ m\in\mathbb{Z},\ 1\leq n\leq N\}$
be a sequence of complex numbers. We write $\lambda_{m,n}=\gamma_{m,n}e^{\delta_{m,n}}e^{i\theta_{m,n}}$, for every $m\in\mathbb{Z}$ and $1\leq n\leq N$, where $\delta_{m,n},\theta_{m,n}\in\mathbb{R}.$ Let 
$$\Lambda:=\Big\{\gamma_{m,n}\in\Gamma_\nu\ :\ (m,n)\in\mathbb{Z}\times\big(\mathbb{Z}\setminus\{1\leq n\leq N\}\big) \Big\}\cup\Lambda_N\ =\ \{\lambda_{m,n}\}_{(m,n)\in\mathbb{Z}^2}.$$ We associate with $\Lambda$  the following infinite product
{\small\[
G_\Lambda(z) := (z-\lambda_{0,0})\underset{m,n\in\mathbb{Z}}{{\prod\ }^{\prime}} \left(1-\frac{z}{\lambda_{m,n}}\right)\exp\left[\frac{z}{\gamma_{m,n}}+\frac{z^2}{2\gamma_{m,n}^2}\right],\quad z\in\mathbb{C}.
\]}
The product (with a prime) is taken over all integers $m$ and $n$ with $(m,n)\neq(0,0)$. The following lemma gives some estimates on $G_\Lambda$.

\begin{lemma}\label{lem1}
$(1)$ Let $\Lambda$ satisfy the conditions of Theorem \ref{thm1}. The infinite product $G_\Lambda$ converges uniformly on every compact set of $\mathbb{C}$ and satisfies the  estimates
{\small
\begin{equation}
\frac{(1+|\textmd{Im} z|)^M}{(1+|z|)^{\nu-\hat{\delta}+M}}\ \mathrm{dist}(z,\Lambda)\ \lesssim\ \left|G_\Lambda(z)\right|e^{-\frac{\pi}{2}|z|^2}\ \lesssim\ \frac{(1+|z|)^{\delta-\nu+M}}{(1+|\textmd{Im} z|)^M} \ \mathrm{dist}(z,\Lambda),\quad z\in\mathbb{C},\nonumber
\end{equation}}
for some positive constant $M$, where $\delta=\delta(\Lambda)+\varepsilon$ and $\hat{\delta}=\hat{\delta}(\Lambda)-\varepsilon$, for some small positive $\varepsilon$.\\
$(2)$ The function $G_{\Lambda_{\beta,\nu}}$ is holomorphic  in $\mathbb{C}$ and satisfies the estimate
{\small\begin{equation}
\left|G_{\Lambda_{\beta,\nu}}(z)\right|e^{-\frac{\pi}{2}|z|^2}\ \asymp\ \frac{\mathrm{dist}(z,\Lambda_{\beta,\nu})}{(1+|z|)^{\nu+2\beta}},\quad z\in\mathbb{C},\nonumber
\end{equation}}
where $\Lambda_{\beta,\nu}:=\Big\{\gamma_{m,n}\in\Gamma_\nu\ :\ (m,n)\in\mathbb{Z}\times\big(\mathbb{Z}\setminus\{1\leq n\leq N\}\big) \Big\}\cup\Lambda_{\beta,N}$ and 
$\Lambda_{\beta,N} :=\left\{m+\mathrm{sign}(m)\frac{\beta}{N}+i n\ :\ m\in\mathbb{Z},\ 1\leq n\leq N\right\}.$
\end{lemma}
\begin{proof}

To prove that $G_\Lambda$ is convergent, it suffices to prove that the series
{\small\begin{eqnarray}\nonumber
\underset{m}{\sum} \log\left|\left(1-\frac{z}{\lambda_{m,n}}\right)e^{\frac{z}{\gamma_{m,n}}+\frac{z^2}{2\gamma_{m,n}^2}}\right|
\end{eqnarray}}
converges uniformly on every compact set of $\mathbb{C}$, for each $1\leq n\leq N$. To this end, fix $r>0$ and $|z|\leq r$. For sufficiently large $m$ we have
\begin{align}
\left|\log\left[\left(1-\frac{z}{\lambda_{m,n}}\right)e^{\frac{z}{\gamma_{m,n}}+\frac{z^2}{2\gamma_{m,n}^2}}\right]\right| & \leq \left| -\frac{z}{\lambda_{m,n}}+\frac{z}{\gamma_{m,n}}\right|+O\left(\frac{|z|^2}{m^2}\right) \nonumber\\
 & \lesssim  |z|\frac{|\delta_{m,n}|+|\theta_{m,n}|}{m}+O\left(\frac{|z|^2}{m^{2}}\right).\nonumber 
\end{align}
The latter inequality together with conditions of Theorem \ref{thm1} imply that $G_\Lambda$ is an entire function and vanishes exactly on $\Lambda$.\\

Now, we estimate the function $G_\Lambda$. To do this, write
\begin{align}
G_\Lambda(z) & = \sigma(z)\frac{z-\gamma_{0,0}}{z}\ \underset{m\in\mathbb{Z}, \ 0\leq n\leq N}{{\prod\ } ^{\prime}} \frac{1-z/\lambda_{m,n}}{1-z/\gamma_{m,n}},\quad z\in\mathbb{C}\setminus \Gamma_N,  \nonumber
\end{align}
where $\Gamma_N:=\{\gamma_{m,n}\in\Gamma_\nu\ :\ (m,n)\in\mathbb{Z}\times\{0\leq n\leq N\} \}$. Taking into account the estimate of the Weierstrass $\sigma-$function, it suffices to prove that there exists a positive number $M$ such that the following estimates hold :
\begin{eqnarray}
\frac{(1+|\textmd{Im} z|)^M}{(1+|z|)^{M+\nu-\hat{\delta}}}\ \lesssim\ \frac{\mathrm{dist}(z,\Gamma_0)}{\mathrm{dist}(z,\Lambda)}\underset{m\in\mathbb{Z}, \ 0\leq n\leq N}{{\prod\ }^{\prime}}\left|\frac{1-z/\lambda_{m,n}}{1-z/\gamma_{m,n}}\right|\lesssim \frac{(1+|z|)^{M-\nu+\delta}}{(1+|\textmd{Im} z|)^M}.\nonumber
\end{eqnarray}
{First of all, it is clear that 
\begin{equation}
\prod_{0\leq n\leq N} \left|\frac{1-z/\lambda_{0,n}}{1-z/\gamma_{0,n}}\right|\ \asymp\ 1,\quad |z|\rightarrow\infty.
\end{equation}}
Now, set
\begin{equation}
\Theta(z):=\underset{m\geq 1, \ 0\leq n\leq N}{\prod}\frac{(1-z/\lambda_{m,n})(1-z/\lambda_{-m,n})}{(1-z/\gamma_{m,n})(1-z/\gamma_{-m,n})}=\Theta_1(z)\Theta_2(z),\quad z\in\mathbb{C}\setminus \Gamma_N, 
\label{Estim2}
\end{equation}
where
{\small
\begin{align}
\Theta_{1}(z) =  \underset{1\leq m\leq 2|z|,\ 0\leq n\leq N}{\prod}\frac{\gamma_{m,n}\gamma_{-m,n}}{\lambda_{m,n}\lambda_{-m,n}}\ \frac{(z-\lambda_{m,n})(z-\lambda_{-m,n})}{(z-\gamma_{m,n})(z-\gamma_{-m,n})},\nonumber
\end{align}}
and
\begin{eqnarray}
\Theta_{2}(z) =\underset{|m|> 2|z|,\ 0\leq n\leq N}{\prod} \frac{1-z/\lambda_{m,n}}{1-z/\gamma_{m,n}}. \end{eqnarray}
{\bf $\bullet$ Estimates of $\Theta_1$.} Let $\delta=\delta(\Lambda)+\varepsilon$ and $\hat{\delta}=\hat{\delta}(\Lambda)-\varepsilon$, where $\varepsilon$ is a small positive real number. We have
\begin{align*}
 \underset{0\leq n\leq N}{\underset{1\leq m\leq 2|z|}{\sum}}\log\left|\frac{\gamma_{m,n}\gamma_{-m,n}}{\lambda_{m,n}\lambda_{-m,n}}\right| & = \underset{1\leq m\leq 2|z|}{\sum} \log\frac{m}{m+\nu} + \underset{1\leq m\leq 2|z|,\ 1\leq n\leq N}{\sum} (\delta_{m,n}+\delta_{-m,n}) \\
     & =  -\nu\log|z|+ \underset{1\leq m\leq 2|z|,\ 1\leq n\leq N}{\sum} (\delta_{m,n}+\delta_{-m,n}) + O(1).
\end{align*}
By the definition of $\delta(\Lambda)$ and $\hat{\delta}(\Lambda)$, we have
{\small\begin{align}\label{11}
(-\nu+\hat{\delta})\log|z|+O(1)\ \leq\ \underset{1\leq m\leq 2|z|,\ 0\leq n\leq N}{\sum}\log\left|\frac{\gamma_{m,n}\gamma_{-m,n}}{\lambda_{m,n}\lambda_{-m,n}}\right|\ \leq\ (-\nu+\delta)\log|z| + O(1).
\end{align}}
Suppose that $z\in\mathbb{C}^+:=\{z\in\mathbb{C}\ :\ \textmd{Re} z\geq 0\}$. We use that
\begin{eqnarray}\label{log}
\log|1+z|=\textmd{Re} z+O(|z|^2),\quad \textmd{Re}(z)\geq -\frac{1}{2},\ |z|<1.
\end{eqnarray}
Now, for fixed $1\leq n\leq N$, we have
{\small\begin{align}
\left|\log\underset{1\leq m\leq 2|z|}{\prod}\left| \frac{z-\lambda_{-m,n}}{z-\gamma_{-m,n}}\right|\right| 
& = \left|\sum_{m\in J} \log\left|1+\frac{\gamma_{m,n}-\gamma_{m,n}e^{\delta_{m,n}}e^{i\theta_{m,n}}}{z-\gamma_{m,n}}\right|\right| \nonumber \\
  & \leq \sum_{m\in J} \left|\frac{\gamma_{m,n}-\gamma_{m,n}e^{\delta_{m,n}}e^{i\theta_{m,n}}}{z-\gamma_{m,n}}\right| + O(1) \nonumber \\
  & \asymp \sum_{m\in J} |m|\frac{\left|\delta_{m,n}\right|+\left|\theta_{m,n}\right|}{|z|} + O(1) = O(1) \nonumber
\end{align}}
where $J:=\{m\in\mathbb{Z}\ :\ -2|z|\leq m\leq -1\}$. Also, we have
$
\underset{1\leq m<2|z|}{\prod}\left|\frac{z-\lambda_{-m,0}}{z-\gamma_{-m,0}}\right| = 1.
$
Therefore, 
\begin{eqnarray}\label{4}
\underset{1\leq m<2|z|,\ 0\leq n\leq N}{\prod}\left|\frac{z-\lambda_{-m,n}}{z-\gamma_{-m,n}}\right|\ \asymp\ 1.
\end{eqnarray}
Now, since the sequences $(m\delta_{m,n})$ and $(m\theta_{m,n})$ are bounded, it follows that $\gamma_{m,n}-\gamma_{m,n}e^{\delta_{m,n}}e^{i\theta_{m,n}}$ is bounded too. Let $j$ be an integer such that $\left|\gamma_{m,n}-\gamma_{m,n}e^{\delta_{m,n}}e^{i\theta_{m,n}}\right|\leq j$. We next write $z=u+iv$. In what follows, we will divide $K:=\{m\in\mathbb{Z}\ :\ 1\leq m\leq 2|z|\}$ into three parts $K=K_0\cup K_1\cup K_2$, where $K_0:=\{m\in\mathbb{N}\ :\ u-10j<m< u+10j\}$, $K_1:=\left\{m\in K\ :\ m\leq u-10j \right\}$ and $K_2:=\left\{m\in K\ :\ u+10j\leq m \right\}$. Clearly, we have
\begin{eqnarray}\label{dist}
\underset{0\leq n\leq N}{\underset{m\in K_0}{{\prod\ }^{\prime}}}\left|\frac{z-\gamma_{m,n}e^{\delta_{m,n}}e^{i\theta_{m,n}}}{z-\gamma_{m,n}}\right|\ \asymp\ \frac{\mathrm{dist}(z,\Lambda)}{\mathrm{dist}(z,\Gamma_\nu)}.
\end{eqnarray}
Next we write $\gamma_{m,n}-\gamma_{m,n}e^{\delta_{m,n}}e^{i\theta_{m,n}} = \eta_{m,n}=|\eta_{m,n}|e^{i\psi_{m,n}}$. Using \eqref{log}, for fixed $1\leq n\leq N$, we get 
{\small\begin{eqnarray}
\log\underset{m\in K_1\cup K_2}{\prod}\left|\frac{z-\gamma_{m,n}e^{\delta_{m,n}}e^{i\theta_{m,n}}}{z-\gamma_{m,n}}\right| 
   & = & \underset{m\in K_1\cup K_2}{\sum} \textmd{Re} \left(\frac{\eta_{m,n}}{z-\gamma_{m,n}}\right) + O(1) \nonumber\\
   & = & \underset{m\in K_1\cup K_2}{\sum}\ |\eta_{m,n}|\cos \left(\psi_{m,n}\right)\frac{u-m}{|z-\gamma_{m,n}|^2} \nonumber\\
   &  & + \underset{m\in K_1\cup K_2}{\sum}\ |\eta_{m,n}|\sin \left(\psi_{m,n}\right)\frac{v-n}{|z-\gamma_{m,n}|^2} + O(1). \nonumber
\end{eqnarray}}
Hence,
{\small\begin{eqnarray}
\left|\underset{m\in K_1\cup K_2}{\sum} |\eta_{m,n}|\sin \left(\psi_{m,n}\right)\frac{v-n}{|z-\gamma_{m,n}|^2} \right|\ & \lesssim\ &\underset{m\in K_1\cup K_2}{\sum} \frac{|v-n|}{(u-m)^2+(v-n)^2}\nonumber\\
    & \leq & \underset{m\geq 1}{\sum} \frac{1}{|v-n|\left[(\frac{u-m}{v-n})^2+1\right] }\ =\  O(1). \nonumber
\end{eqnarray}}
Thus,
{\small\begin{align}
\left|\log\underset{m\in K_1\cup K_2}{\prod} \left|\frac{z-\gamma_{m,n}e^{\delta_{m,n}}e^{i\theta_{m,n}}}{z-\gamma_{m,n}}\right|\right| & =  \left|\underset{m\in K_1\cup K_2}{\sum}\frac{u-m}{|z-\gamma_{m,n}|^2}|\eta_{m,n}|\cos(\psi_{m,n}) + O(1)\right| \nonumber \\
 & \leq j\underset{n\in K_1}{\sum}\frac{u-m}{|z-\gamma_{m,n}|^2}+j\underset{m\in K_2}{\sum}\frac{m-u}{|z-\gamma_{m,n}|^2} + O(1) \nonumber \\
 & = M\log\frac{|z|}{1+|\textmd{Im} z|} + O(1). \label{3.14}
\end{align}}
It follows from \eqref{dist} and \eqref{3.14} that
{\small\begin{equation}\label{6}
\left(\frac{1+|\textmd{Im} (z)|}{1+|z|}\right)^{M}\frac{\mathrm{dist}(z,\Lambda)}{\mathrm{dist}(z,\Gamma_0)}\lesssim\underset{1\leq n\leq N}{\underset{m\in K}{\prod}}\left|\frac{z-\gamma_{m,n}e^{\delta_{m,n}}e^{i\theta_{m,n}}}{z-\gamma_{m,n}}\right|
 \lesssim \left(\frac{1+|z|}{1+|\textmd{Im} (z)|}\right)^{M}\frac{\mathrm{dist}(z,\Lambda)}{\mathrm{dist}(z,\Gamma_0)}.
 \end{equation}}
When $n=0$, we have
\begin{align}
\log\prod_{m\in K_1\cup K_2} \left|\frac{z-\lambda_{m,n}}{z-\gamma_{m,n}}\right|  & = \sum_{m\in K_1\cup K_2}\log\left|1-\frac{\nu}{z-m}\right|\nonumber\\ 
   & = -\nu\underset{m\in K_1\cup K_2}{\sum} \frac{u-m}{(u-m)^2+v^2} + O(1) \ =\ O(1). \label{key44}
\end{align}
Combining \eqref{11}, \eqref{4}, \eqref{6}  and \eqref{key44}, we conclude that
\begin{equation}\label{theta1}
\frac{(1+|\textmd{Im} z|)^M}{(1+|z|)^{\nu-\hat{\delta}+M}}\,\frac{\mathrm{dist}(z,\Lambda)}{\mathrm{dist}(z,\Gamma)}\,\ \lesssim\,\ \left|\Theta_1(z)\right|\ \lesssim\ \frac{(1+|z|)^{\delta-\nu+M}}{(1+|\textmd{Im} z|)^M} \,\frac{\mathrm{dist}(z,\Lambda)}{\mathrm{dist}(z,\Gamma)},
\end{equation}
for every $z\in \mathbb{C}\setminus\Gamma$.\\

$\bullet$ {\bf Estimates of $\Theta_2$.} Under the conditions of Theorem \ref{thm1}, for fixed $0\leq n\leq N$, we have
{\small\begin{align}
\log\left|\Theta_2(z)\right| & =  \underset{m> 2|z|}{\sum}\log\left|\frac{(1-z/\lambda_{m,n})(1-z/\lambda_{-m,n})}{(1-z/\gamma_{m,n})(1-z/\gamma_{-m,n})}\right| \nonumber \\
&  \leq  \underset{m>2|z|}{\sum}\left|\frac{1-z/\lambda_{m,n}}{1-z/\gamma_{m,n}}-1\right| + \left|\frac{1-z/\lambda_{-m,n}}{1+z/\gamma_{m,n}}-1\right| \nonumber\\
   & \lesssim  |z|\underset{m> 2|z|}{\sum}\frac{|\delta_{m,n}|+|\theta_{m,n}|}{m}\  =\ O(1). \nonumber
\end{align}}
By changing the roles of $\Gamma_0$ and $\Lambda$ in the above calculation, we obtain
\begin{eqnarray}
\underset{0\leq n\leq N}{ \underset{m> 2|z|}{\prod}} \left|\frac{(1-z/\lambda_{m,n})(1-z/\lambda_{-m,n})}{(1-z/\gamma_{m,n})(1-z/\gamma_{-m,n})}\right|\ \asymp\ 1. \label{.3}
\end{eqnarray}
Finally, combining \eqref{theta1} and \eqref{.3} we get
{\small\begin{equation}
\frac{(1+|\textmd{Im} z|)^M}{(1+|z|)^{\nu-\hat{\delta}+M}}\ \mathrm{dist}(z,\Lambda)\ \lesssim\ \left|G_\Lambda(z)\right|e^{-\frac{\pi}{2}|z|^2}\ \lesssim\ \frac{(1+|z|)^{\delta-\nu+M}}{(1+|\textmd{Im} z|)^M} \ \mathrm{dist}(z,\Lambda),\quad z\in\mathbb{C}.\nonumber
\end{equation}}

$(2)$ In order to estimate the function $G_{\Lambda_{\beta,\nu}}$, we write
{\small\begin{eqnarray}
G_{\Lambda_{\beta,\nu}}(z)\ =\ G_{\Gamma_\nu}(z) \underset{1\leq n\leq N}{\underset{m\in\mathbb{Z}}{\prod}} \frac{1-z/\lambda_{m,n}}{1-z/\gamma_{m,n}}\ =\ G_{\Gamma_\nu}(z)\Theta_1(z)\Theta_2(z), \quad z\in\mathbb{C}\setminus \Gamma_N,\nonumber
\end{eqnarray}}
where $\Theta_1$ and $\Theta_2$ are as in the previous argument. As above, \begin{eqnarray}\label{key1}
\left|\Theta_2(z)\right| \ \asymp\ 1,\quad z\in\mathbb{C}. 
\end{eqnarray}
To estimate the factor $\Theta_1$, we observe that
\begin{equation}\label{key2}
\underset{1\leq n\leq N}{\underset{1\leq m\leq 2|z|}{\sum}} \log\left|\frac{\gamma_{m,n}\gamma_{-m,n}}{\lambda_{m,n}\lambda_{-m,n}}\right| 
 =  -2\beta \underset{1\leq m\leq 2|z|}{\sum}\frac{1}{m} + O(1)=  -2\beta\log|z| + O(1).
\end{equation}
Now suppose that $\textmd{Re} z>0$. As above, we have
\begin{eqnarray}\label{key3}
\prod_{m\leq 2|z|}\left|\frac{z-\lambda_{-m,n}}{z-\gamma_{-m,n}}\right|\ \asymp\ 1,\quad z\in\mathbb{C}\setminus \Gamma_N.
\end{eqnarray}
Finally,
{\small\begin{align}
\log\prod_{1\leq m\leq 2|z|,n}\left|\frac{z-\lambda_{m,n}}{z-\gamma_{m,n}}\right| 
   & =  \log\underset{m\in K_0,n}{\prod}\left|\frac{z-\lambda_{m,n}}{z-\gamma_{m,n}}\right|+ \underset{m\in K_1\cup K_2,n}{\sum}\log\left|1-\frac{\beta/N}{z-\gamma_{m,n}}\right| \nonumber\\
   & = \log\frac{\mathrm{dist}(z,\Lambda_{\beta,\nu})}{\mathrm{dist}(z,\Gamma)}-2\beta \log\frac{|1-z|}{10j+|v|}\frac{10j+|v|}{|2|z|-z|} + O(1) \nonumber\\
   & = \log\frac{\mathrm{dist}(z,\Lambda_{\beta,\nu})}{\mathrm{dist}(z,\Gamma)}+O(1),\label{key4}
\end{align}}
where $z=u+iv$. By combining \eqref{key1}, \eqref{key2}, \eqref{key3} and \eqref{key4}, we obtain
the desired estimate. This completes the proof.
\end{proof}}

Now, let $\Lambda:=\{\lambda_{m,n}\ :\ m,n\in\mathbb{Z}\}$ be a sequence of complex numbers. We write $\lambda_{m,n}=\gamma_{m,n}e^{\delta_{m,n}}e^{i\theta_{m,n}}$, for every $m,n\in\mathbb{Z}$, where $\delta_{m,n},\theta_{m,n}\in\mathbb{R}.$ Set 
\[
G_\Lambda(z) := (z-\lambda_{0,0})\underset{(m,n)\in\mathbb{Z}^2}{{\prod\ }^{\prime} } \left(1-\frac{z}{\lambda_{m,n}}\right)\exp\left[\frac{z}{\gamma_{m,n}}+\frac{z^2}{2\gamma_{m,n}^2}\right],\quad z\in\mathbb{C}.
\]

\begin{lemma}\label{lem1full}
$(1)$ Under the conditions of Theorem \ref{thm1full}, $G_\Lambda$ is an entire function and we have
{\small\begin{equation}
\frac{(1+|\textmd{Im} z|)^M}{(1+|z|)^{\nu-\hat{\delta}+M}}\ \mathrm{dist}(z,\Lambda)\ \lesssim\ \left|G_\Lambda(z)\right|e^{-\frac{\pi}{2}|z|^2}\ \lesssim\ \frac{(1+|z|)^{\delta-\nu+M}}{(1+|\textmd{Im} z|)^M} \ \mathrm{dist}(z,\Lambda),\quad z\in\mathbb{C},\nonumber
\end{equation}}
for some positive constant $M$, where $\delta=\delta(\Lambda)+\varepsilon$ and $\hat{\delta}=\hat{\delta}(\Lambda)-\varepsilon$, for some small positive $\varepsilon$. \\
$(2)$ The function $G_{\widetilde{\Lambda}_{\beta,\nu}}$ is holomorphic  in $\mathbb{C}$ and satisfies the estimate
{\small\begin{equation}
\frac{(1+|\textmd{Im} z|)^M}{(1+|z|)^{\nu+\beta+M}}\ \mathrm{dist}(z,\widetilde{\Lambda}_{\beta,\nu})\ \lesssim\ \left|G_{\widetilde{\Lambda}_{\beta,\nu}}(z)\right|e^{-\frac{\pi}{2}|z|^2}\ \lesssim\ \frac{(1+|z|)^{-\beta-\nu+M}}{(1+|\textmd{Im} z|)^M} \ \mathrm{dist}(z,\widetilde{\Lambda}_{\beta,\nu}),\nonumber
\end{equation}}
for every $z\in\mathbb{C}.$
\end{lemma}
\begin{proof}
The infinite product defining $G_\Lambda$ converges uniformly on every compact set of $\mathbb{C}$. Indeed, we have
{\small
\begin{align}
\left|\log\left[\left(1-\frac{z}{\lambda_{m,n}}\right)e^{\frac{z}{\gamma_{m,n}}+\frac{z^2}{2\gamma_{m,n}^2}}\right]\right| & \leq \left| -\frac{z}{\lambda_{m,n}}+\frac{z}{\gamma_{m,n}}\right|+ \left| -\frac{z^2}{2\lambda_{m,n}^2}+\frac{z^2}{2\gamma_{m,n}^2}\right| + O\left(\frac{|z|^3}{|\gamma_{m,n}|^3}\right) \nonumber\\
 & \lesssim  |z|\frac{|\delta_{m,n}|+|\theta_{m,n}|}{|\gamma_{m,n}|} + |z|\frac{|\delta_{m,n}|+|\theta_{m,n}|}{|\gamma_{m,n}|^2} + O\left(\frac{|z|^3}{|\gamma_{m,n}|^3}\right).\nonumber
\end{align}}
Now, to estimate the function $G_\Lambda$ we first write
\[
\frac{|G_\Lambda(z)|}{\mathrm{dist}(z,\Lambda)}\ =\ \frac{|G_{\Gamma_\nu}(z)|}{\mathrm{dist}(z,\Gamma_\nu)}\prod_{m,n}\left|\frac{1-z/\lambda_{m,n}}{1-z/\gamma_{m,n}}\right|\ \frac{\mathrm{dist}(z,\Gamma_\nu)}{\mathrm{dist}(z,\Lambda)}.
\]
We have
\begin{eqnarray}
\sum_{m,n\in\mathbb{Z}}\log\left|\frac{1-z/\lambda_{m,n}}{1-z/\gamma_{m,n}}\right| & = & \sum_{|\gamma_{m,n}|\leq 2|z|} \log\left|\frac{1-z/\lambda_{m,n}}{1-z/\gamma_{m,n}}\right| + \sum_{|\gamma_{m,n}|>2|z|} \cdots.  \nonumber
\end{eqnarray}
As in the proof of Lemma \ref{lem1} we have 
{\small\begin{align}
\sum_{|\gamma_{m,n}|> 2|z|} \log\left|\frac{1-z/\lambda_{m,n}}{1-z/\gamma_{m,n}}\right| & \leq \sum_{m^2+n^2> 4|z|^2} \left|\frac{z/\gamma_{m,n}-z/\lambda_{m,n}}{1-z/\gamma_{m,n}}\right| \nonumber \\
   &
    \lesssim |z|\ \sum_{m^2+n^2> 4|z|^2}\frac{|\delta_{m,n}|+|\theta_{m,n}|}{(m^2+n^2)^{1/2}} = O(1). \nonumber
\end{align}}
By changing the roles of
 $\Lambda$ and $\Gamma_\nu$ in the previous calculation,  we conclude that
\begin{eqnarray}
\prod_{|\gamma_{m,n}|\geq 2|z|}\ \left|\frac{1-z/\lambda_{m,n}}{1-z/\gamma_{m,n}}\right|\ \asymp\ 1.\label{keyI_2}
\end{eqnarray} 
Furthermore, we have
\begin{align}
\sum_{|\gamma_{m,n}|\leq 2|z|} \log\left|\frac{1-z/\lambda_{m,n}}{1-z/\gamma_{m,n}}\right| & = \sum_{|\gamma_{m,n}|\leq 2|z|}\delta_{m,n} + \sum_{|\gamma_{m,n}|\leq 2|z|} \log\left|\frac{z-\lambda_{m,n}}{z-\gamma_{m,n}}\right|.\nonumber
\end{align}
By the definitions of $\delta(\Lambda)$ and $\hat{\delta}(\Lambda)$, we have
\begin{eqnarray}
\hat{\delta}\log|z|+O(1)\ \leq\ \sum_{|\gamma_{m,n}|\leq 2|z|}\delta_{m,n}\ \leq\ \delta\log|z| + O(1).
\end{eqnarray}
If $\textmd{Re} z>0$, then
{\small\begin{align}
\underset{m\leq 0}{\underset{m^2+n^2\leq 4|z|^2}{\sum}}\ \log\left|1-\frac{\gamma_{m,n}(1-e^{\delta_{m,n}}e^{i\theta_{m,n}})}{z-\gamma_{m,n}}\right| & \leq \sum_{m^2+n^2\leq 4|z|^2,\ m\leq 0}\ \left|\frac{\gamma_{m,n}(1-e^{\delta_{m,n}}e^{i\theta_{m,n}})}{z-\gamma_{m,n}}\right|\nonumber\\
   & \lesssim \frac{1}{|z|} \sum_{0<m^2+n^2\leq 4|z|^2}\ \frac{1}{(m^2+n^2)^{1/2}}\ =\ O(1). \nonumber
\end{align}}
Hence,
{\small\begin{eqnarray}
\log\prod_{|\gamma_{m,n}|\leq 2|z|}\left|\frac{z-\lambda_{m,n}}{z-\gamma_{m,n}}\right| & = & \sum_{|\gamma_{m,n}|\leq 2|z|,\ m>0} \log\left|\frac{z-\lambda_{m,n}}{z-\gamma_{m,n}}\right| \nonumber\\
    & = & \log\frac{\mathrm{dist}(z,\Lambda)}{\mathrm{dist}(z,\Gamma_\nu)} + \sum_{|\gamma_{m,n}|\leq |z|/2} \log\left|\frac{z-\lambda_{m,n}}{z-\gamma_{m,n}}\right|\nonumber\\
   & & + \sum_{|z|/2\leq |\gamma_{m,n}|\leq 2|z|,\ m>0} \cdots + O(1). \nonumber
\end{eqnarray}}
Since
{\small\begin{eqnarray}
 \sum_{|\gamma_{m,n}|\leq |z|/2} \log\left|\frac{z-\lambda_{m,n}}{z-\gamma_{m,n}}\right| & \leq &  \sum_{|\gamma_{m,n}|\leq |z|/2} \left| \frac{\gamma_{m,n}-\lambda_{m,n}}{z-\gamma_{m,n}}\right|\nonumber\\
 & \lesssim & \frac{1}{|z|}\sum_{m^2+n^2\leq |z|^2/4}\frac{1}{(m^2+n^2)^{1/2}} = O(1), \nonumber
\end{eqnarray}}
we obtain
{\small\begin{eqnarray}
&&  \log\prod_{|\gamma_{m,n}|\leq 2|z|}\left|\frac{z-\lambda_{m,n}}{z-\gamma_{m,n}}\right| \nonumber\\
& = & \log\frac{\mathrm{dist}(z,\Lambda)}{\mathrm{dist}(z,\Gamma_\nu)} + \sum_{|z|/2\leq |\gamma_{m,n}|\leq 2|z|,\ m>0} \textmd{Re}\left[\frac{\gamma_{m,n}-\lambda_{m,n}}{z-\gamma_{m,n}}\right] + O(1) \nonumber\\
      & = & \log\frac{\mathrm{dist}(z,\Lambda)}{\mathrm{dist}(z,\Gamma_\nu)} + {\sum}^{\prime\prime} \textmd{Re}\left[\frac{|\eta_{m,n}|e^{i\psi_\gamma}}{\gamma(z-\gamma_{m,n})}\right] + O(1) \nonumber\\
     & = & \log\frac{\mathrm{dist}(z,\Lambda)}{\mathrm{dist}(z,\Gamma_\nu)} + {\sum}^{\prime\prime} \textmd{Re}\left[\frac{|\eta_{m,n}|e^{i\theta_\gamma+i\psi_\gamma}((u-m)-i(v-n))}{|\gamma||z-\gamma_{m,n}|^2}\right] + O(1), \nonumber
\end{eqnarray}}
the summation (with double primes) is taken over all integers $m$ and $n$ with $(m,n)\in\Big\{(m,n)\in\mathbb{Z}_+\times\mathbb{Z}\ :\ |z-\lambda_{m,n}|,|z-\gamma_{m,n}|\geq 1,\ \mbox{and}\ |z|/2\leq |\gamma_{m,n}|\leq 2|z|\Big\}$.  Note again that
{\small\begin{eqnarray}
\left|{\sum}^{\prime\prime}\frac{|\eta_\gamma|(v-n)\sin(\phi)}{|\gamma|\left((u-m)^2+(v-n)^2\right)}\right| \leq \frac{A}{|z|}\sum_{-2|z|\leq n\leq 2|z|}\sum_{m}  \frac{1}{|v-n|\left(\left(\frac{u-m}{v-n}\right)^2+1\right)} = O(1).\nonumber
\end{eqnarray}}
Consequently,
{\small\begin{eqnarray}
\left|\log{\prod}^{\prime\prime}\left|\frac{z-\lambda_{m,n}}{z-\gamma_{m,n}}\right| \right|& \leq &  \frac{A}{|z|}{\sum}^{\prime\prime} \frac{|u-m|}{|z-\gamma_{m,n}|^2} + O(1) \nonumber \\
  & \leq &  \frac{A}{|z|}\sum_{-2|z|\leq n\leq 2|z|}\sum_m \frac{|u-m|}{|z-\gamma_{m,n_v}|^2} + O(1) \nonumber \\
  & = & M \log\frac{1+|z|}{1+|\textmd{Im} z|}.
\end{eqnarray}}
Finally,
\begin{eqnarray}
\frac{(1+|\textmd{Im} z|)^M}{(1+|z|)^{\nu-\hat{\delta}+M}}\mathrm{dist}(z,\Lambda)\ \lesssim\ |G_\Lambda(z)|e^{-\frac{\pi}{2}|z|^2}\ \lesssim\ \frac{(1+|z|)^{\delta-\nu+M}}{(1+|\textmd{Im} z|)^M}\mathrm{dist}(z,\Lambda).\nonumber
\end{eqnarray}
$(2)$ The estimate of the function $G_{\widetilde{\Lambda}_{\beta,\nu}}$ can be obtained by the same way. We just write 
{\small\begin{eqnarray}
\log\left|\frac{G_{\widetilde{\Lambda}_{\beta,\nu}}(z)}{G_{\Lambda_\nu}(z)}\right| & = & \sum_{|\gamma|\leq 2|z|}+\sum_{|\gamma|\geq 2|z|}\ \log\left|\frac{1-z/\lambda_\gamma}{1-z/\gamma}\right| \nonumber\\
    & = & \sum_{|\gamma|\leq 2|z|} \frac{-\beta}{|\gamma|^2}+\log\left|\frac{z-\gamma e^{\frac{\beta}{|\gamma|^2}}}{z-\gamma}\right|+ O(1) \nonumber \\
    & = & -\beta\log|z| + \sum_{|\gamma|\leq 2|z|} \log\left|\frac{z-\gamma e^{\frac{\beta}{|\gamma|^2}}}{z-\gamma}\right|+ O(1). \nonumber
\end{eqnarray}}
By the above argument and the estimates on the function $G_{\Gamma_\nu}$, the desired estimates follow.
\end{proof}

In the rest of this section, we prepare some  ingredients necessary for proving the results in the third subsection. Let $G_\Gamma$ be the entire function defined by 
\begin{eqnarray}\label{R}
G_\Gamma(z) :=  \frac{z^2-1}{\pi z^2}\ \sin\left(\frac{\pi}{2} z^2\right),\quad z\in\mathbb{C}, \nonumber
\end{eqnarray}
and vanishing exactly on $\Gamma=\{\pm\sqrt{2n},\ \pm i\sqrt{2n}\ :\ n\geq 1\}\cup\{\pm 1\}$. In a similar way, we associate the function
\begin{equation}\label{F1}
G_\Lambda(z) = (z^2-1)\underset{n\geq 1}{\prod}\left(1+\frac{z^2}{2n}\right)\left(1+\frac{z}{\sqrt{2n}}\right)\left(1-\frac{z}{\lambda_n}\right), \quad z\in\mathbb{C},
\end{equation}
to   $\Lambda:=\{\lambda_n,\ -\sqrt{2n},\ \pm i\sqrt{2n}\ :\ n\geq 1\}\cup\{\pm 1\}$. Furthermore, we consider the functions $G_{\Lambda_\beta}$ given in \eqref{F1} associated to \begin{eqnarray}\label{lambda-beta}
\Lambda_\beta:=\Big\{ \sqrt{2n+4\beta},\ -\sqrt{2n},\ \pm i\sqrt{2n}\ :\ n\geq 1 \Big\}\cup\left\{\pm  1\right\}.
\end{eqnarray}
We denote $\mathcal{R}:=\left\{z\in\mathbb{C}\ :\ \textmd{Re} z > |\textmd{Im} z|\geq 0 \right\}$. The following  key lemma describes some properties of the functions $G_\Lambda$ and $G_{\Lambda_\beta}$.

\begin{lemma}\label{lem1ALS}
$(1)$ Under the conditions of Theorem \ref{thm1ALS}, the infinite product in \eqref{F1} converges uniformly on every compact set of $\mathbb{C}$ and hence $G_\Lambda$ is an entire function vanishing exactly on $\Lambda$. Furthermore,
{\small\begin{align}\nonumber
\frac{(1+|\textmd{Im} z^2|)^M\left|G_\Gamma(z)\right|}{(1+|z|)^{2\delta+2M}}\frac{\mathrm{dist}(z,\Lambda)}{\mathrm{dist}(z,\Gamma)}\ \lesssim\ \left|G_\Lambda(z)\right|\ \lesssim\ \frac{(1+|z|)^{2\delta+2M}\left|G_\Gamma(z)\right|}{(1+|\textmd{Im} z^2|)^M}\frac{\mathrm{dist}(z,\Lambda)}{\mathrm{dist}(z,\Gamma)},
\end{align}}
for every $z\in\mathcal{R}$ and
\begin{align}\nonumber
\frac{\left|G_\Gamma(z)\right|}{(1+|z|)^{2\delta}}\ \lesssim\ \left|G_\Lambda(z)\right|\ \lesssim\ (1+|z|)^{2\delta}\left|G_\Gamma(z)\right|,\quad z\in\mathbb{C}\setminus\mathcal{R}.
\end{align}
Here $\delta=\delta(\Lambda)+\varepsilon$ for a positive number $\varepsilon$ small enough.\\
$(2)$  The function $G_{\Lambda_\beta}$ is holomorphic in $\mathbb{C}$ and $\Lambda_\beta$ is exactly the zero set of $G_{\Lambda_\beta}$. Moreover, we have
\begin{eqnarray}\label{rmk}
\left|G_{\Lambda_\beta}(z)\right| \asymp \ \frac{|G_\Gamma(z)|}{(1+|z|)^{2\beta}}\ \frac{\mathrm{dist}(z,\Lambda_\beta)}{\mathrm{dist}(z,\Gamma)},\quad z\in\mathbb{C}\setminus \Gamma,
\end{eqnarray}
and $\Lambda_\beta$ is the zero set of $G_{\Lambda_\beta}$.
\end{lemma}
\begin{proof}

$(1)$ We write the infinite product defining the function $G_\Lambda$ in \eqref{F1} as follows :
{\small\begin{align}
G_\Lambda(\sqrt{2} z)& = (2z^2-1)\underset{n\geq 1}{\prod}\left(1+\frac{z^2}{n}\right)e^{-\frac{z^2}{n}}\times\underset{n\geq 1}{\prod}\left(1+\frac{z}{\sqrt{n}}\right)\left(1-\frac{\sqrt{2}z}{\lambda_n}\right)e^{-\frac{z}{\sqrt{n}}+\frac{z^2}{n}+\frac{z}{\sqrt{n}}} . \label{prod}
\end{align}}

Fix $r>0$ and $|z|\leq r$. For sufficiently large $n$ we have
{\small\begin{align}
\left|\log\left[\left(1-\frac{\sqrt{2}z}{\lambda_n}\right)e^{\frac{z}{\sqrt{n}}+\frac{z^2}{2n}}\right]\right| & \leq \left| -\frac{\sqrt{2}z}{\lambda_n}-\frac{z^2}{\lambda^2_{n}}+\frac{z}{\sqrt{n}}+\frac{z^2}{2n}\right|+O\left(\frac{|z|^3}{n^{3/2}}\right) \nonumber\\
 & \lesssim  |z|\frac{|\delta_n|+|\theta_n|}{\sqrt{n}}+|z|^2\frac{|\delta_n|+|\theta_n|}{n}+\frac{|z|^3}{n^{3/2}} .\nonumber %\\
\end{align}}
Consequently, $G_\Lambda$ is an entire function and vanishes exactly on $\Lambda$. Now, we will estimate the function $G_\Lambda$. To this end, write
\begin{align}
G_\Lambda(\sqrt{2} z) & = \frac{(2z^2-1)}{\pi z^2}\sin(\pi z^2)\ \underset{n\geq 1}{\prod}\frac{1-(\sqrt{2} z)/\lambda_n}{1-z/\sqrt{n}},\quad z\in\mathbb{C}\setminus \Gamma^+,  \nonumber
\end{align}
where $\Gamma^+:=\{\sqrt{n}\ :\ n\geq 1\}$. It suffices to estimate the infinite product
\begin{eqnarray}
\varPsi(z):=\underset{n\geq 1}{\prod}\frac{1-(\sqrt{2} z)/\lambda_n}{1-z/\sqrt{n}}=\varPsi_1(z)\varPsi_2(z),\quad z\in\mathbb{C}\setminus \Gamma^+, \label{Estim2ALS}
\end{eqnarray}
where
\begin{align}
\varPsi_1(z) = \underset{\sqrt{n}\leq 2|z|}{\prod}\frac{\sqrt{2n}}{\lambda_n}\ \frac{z-\lambda_n/\sqrt{2}}{z-\sqrt{n}}\ \ \mbox{and}\quad
\varPsi_2(z) =\underset{\sqrt{n}\geq 2|z|}{\prod}\frac{\sqrt{2n}}{\lambda_n}\ \frac{z-\lambda_n/\sqrt{2}}{z-\sqrt{n}}. \nonumber
\end{align}\\

{\bf $\bullet$ Upper and lower estimates of $\varPsi_1$.} Let $\delta=\delta(\Lambda)+\varepsilon$, where $\varepsilon$ is a small positive real number. We have
\begin{align}\label{11ALS}
\left|\underset{\sqrt{n}\leq 2|z|}{\sum}\log\left|\frac{\sqrt{2n}}{\lambda_n}\right|\right|  \leq \delta \log\left(4|z|^2\right) = 2\delta\log|z| + O(1).
\end{align}
On the one hand, let $z\in\mathbb{C}\setminus\mathcal{R}=\{z\in\mathbb{C}\ :\ \textmd{Re} z\leq 0\ \ \mbox{or}\ \ 0\leq \textmd{Re} z\leq |\textmd{Im} z|\}$. Recall that 
\begin{eqnarray}\label{logALS}
\log|1+z|=\textmd{Re} z+O(|z|^2),\quad \textmd{Re}(z)\geq -\frac{1}{2},\ |z|<1.
\end{eqnarray}
Then
{\small\begin{align}
\left|\log\underset{1\leq\sqrt{n}\leq 2|z|}{\prod}\left| \frac{z-\lambda_n/\sqrt{2}}{z-\sqrt{n}}\right|\right| &
  = \left|\sum_{n\in J} \textmd{Re}\frac{\sqrt{n}-\sqrt{n}e^{\delta_n}e^{i\theta_n}}{z-\sqrt{n}} + O(1)\right| \nonumber \\
  & \asymp \sum_{n\in J} \sqrt{n}\frac{\left|\delta_n\right|+\left|\theta_n\right|}{|z|} + O(1) = O(1), \nonumber
\end{align}}
where $J:=\{n\in\mathbb{N}\ :\ 1\leq n\leq 4|z|^2\}$. Therefore, \begin{eqnarray}\label{4ALS}
\underset{\sqrt{n}<2|z|}{\prod}\left|\frac{z-(\lambda_n/\sqrt{2})}{z-\sqrt{n}}\right|=\underset{ n\in J}{\prod}\left|\frac{z-(\lambda_{n}/\sqrt{2})}{z-\sqrt{n}}\right|\asymp 1.
\end{eqnarray}
Combining \eqref{11ALS} and \eqref{4ALS}, we obtain
\begin{equation}\label{one}
\frac{1}{(1+|z|)^{2\delta}}\ \lesssim\,\ \left|\varPsi_1(z)\right|\ \lesssim\ (1+|z|)^{2\delta},\quad z\in\mathbb{C}\setminus\mathcal{R}.
\end{equation}

On the other hand, let $z\in\mathcal{R}\setminus\Gamma^+$. Using the calculation in the previous case and changing $z$ by $-z$, we get
\begin{equation}\label{pont}
\underset{n\in J}{\prod}\left|\frac{z-(\lambda_n/\sqrt{2})}{z-\sqrt{n}}\right|\asymp \underset{n\in J}{\prod}\left|\frac{z^2-\lambda^2_{n}/2}{z^2-n}\right|=\underset{n\in J}{\prod}\left|\frac{w-ne^{2\delta_{n}}e^{2i\theta_{n}}}{w-n}\right|.
\end{equation}
By the proof of Lemma \ref{lem1} we have
{\small\begin{equation}\label{6ALS}
\left(\frac{1+|\textmd{Im} (z^2)|}{1+|z|^2}\right)^{M}\frac{\mathrm{dist}(z,\Lambda)}{\mathrm{dist}(z,\Gamma^+)}\lesssim\underset{n\in J}{\prod}\left|\frac{w-ne^{2\delta_{n}}e^{2i\theta_{n}}}{w-n}\right|
 \lesssim \left(\frac{1+|z|^2}{1+|\textmd{Im} (z^2)|}\right)^{M}\frac{\mathrm{dist}(z,\Lambda)}{\mathrm{dist}(z,\Gamma^+)}.
 \end{equation}}
Consequently,
{\small\begin{equation}\label{3'}
\left(\frac{1+|\textmd{Im} (z^2)|}{1+|z|^2}\right)^{M}\frac{\mathrm{dist}(z,\Lambda)}{\mathrm{dist}(z,\Gamma^+)}\lesssim\underset{1<\sqrt{n}<2|z|}{\prod}\left|\frac{z-(\lambda_n/\sqrt{2})}{z-\sqrt{n}}\right|\lesssim \left(\frac{1+|z|^2}{1+|\textmd{Im} (z^2)|}\right)^{M}\frac{\mathrm{dist}(z,\Lambda)}{\mathrm{dist}(z,\Gamma^+)}.
\end{equation}}
By \eqref{11ALS} and \eqref{3'}, we get finally the estimate
\begin{equation}\label{theta1ALS}
\frac{(1+|\textmd{Im} z^2|)^M}{(1+|z|)^{2\delta+2M}}\,\frac{\mathrm{dist}(z,\Lambda)}{\mathrm{dist}(z,\Gamma)}\,\ \lesssim\,\ \left|\varPsi_1(z)\right|\ \lesssim\ \frac{(1+|z|)^{2\delta+2M}}{(1+|\textmd{Im} z^2|)^M} \,\frac{\mathrm{dist}(z,\Lambda)}{\mathrm{dist}(z,\Gamma)},
\end{equation}
for every $z\in \mathcal{R}\setminus\Gamma$.\\

$\bullet$ {\bf Upper and lower estimates of $\varPsi_2$.} Under the conditions of Theorem \ref{thm1}, we have
{\small\begin{align}
\log\left|\varPsi_2(z)\right| 
&  \leq  \underset{\sqrt{n}\geq 2|z|}{\sum}\left|\frac{1-(\sqrt{2} z)/\lambda_n}{1-z/\sqrt{n}}-1\right| \nonumber 
    \lesssim  |z|\underset{\sqrt{n}\geq 2|z|}{\sum}\frac{|\delta_n|+|\theta_n|}{\sqrt{n}}\  =\ O(1). \nonumber
\end{align}}
Since $\Gamma$ and $\Lambda$ play symmetric roles in the above calculation, we then get
\begin{eqnarray}
\left|\varPsi_2(z)\right| =  \underset{\sqrt{n}\geq 2|z|}{\prod}\left|\frac{1-(\sqrt{2} z)/\lambda_n}{1-z/\sqrt{n}}\right| \asymp 1. \label{.3ALS}
\end{eqnarray}
Now, combining \eqref{one}, \eqref{theta1}  and \eqref{.3ALS}, we obtain
\begin{eqnarray}
\frac{|G_\Gamma(z)|}{(1+|z|)^{2\delta}}\,\  \lesssim\,\ \left|G_\Lambda(z)\right|\,\ \lesssim\,\ (1+|z|)^{2\delta}|G_\Gamma(z)|,\quad z\in\mathbb{C}\setminus\mathcal{R}, \nonumber
\end{eqnarray}
and \begin{eqnarray}
|G_\Gamma(z)|\frac{(1+|\textmd{Im} z^2|)^M}{(1+|z|)^{2\delta+2M}}\,\frac{\mathrm{dist}(z,\Lambda)}{\mathrm{dist}(z,\Gamma)}\,\  \lesssim\,\ \left|G_\Lambda(z)\right|\,\ \lesssim\,\ |G_\Gamma(z)|\frac{(1+|z|)^{2\delta+2M}}{(1+|\textmd{Im} z^2|)^M}\,\frac{\mathrm{dist}(z,\Lambda)}{\mathrm{dist}(z,\Gamma)}, \nonumber
\end{eqnarray}
for every $z\in\mathcal{R}\setminus \Gamma$, and the first part in Lemma \ref{lem1} is proved.\\

$(2)$ The estimate of the function $G_{\Lambda_\beta}$ is similar. 
\end{proof}

To prove results in the third part, we need the following standard  lemma. For completeness, we include its proof.

\begin{lemma}\label{7.3}
Let $G_{\Gamma}$ be the entire function associated with $\Gamma$ and let $\alpha$ and $\beta$ be two real numbers. The function $G_{\Gamma}$ belongs to  $ L^2\left(\left(\frac{1+|z^2|}{1+|\textmd{Im} z^2|}\right)^\alpha\frac{d\mu}{1+|z|^{2\beta}}\right)$ if and only if $\beta>1/2$.
\end{lemma}
\begin{proof}
We recall that
$$|\sin(z)|^2=\big(\sin(\textmd{Re} z)\big)^2+\big(\mathrm{sh}(\textmd{Im} z)\big)^2,\quad z\in\mathbb{C}.$$
Therefore, the function $G_{\Gamma}$ belongs to  $ L^2\left(\left(\frac{1+|z^2|}{1+|\textmd{Im} z^2|}\right)^\alpha\frac{d\mu}{1+|z|^{2\beta}}\right)$ if and only if
$$\int_\mathbb{C}\left|\frac{\mathrm{sh}(\frac{\pi}{2}\textmd{Im} z^2)}{1+|z|^\beta}\right|^2\left(\frac{1+|z^2|}{1+|\textmd{Im} z^2|}\right)^\alpha e^{-\pi|z|^2}dA(z)\ <\ \infty.$$
By the Tonelli theorem, we get
{\small\begin{eqnarray}
I & := & \int_{|z|>1}\left|\frac{\mathrm{sh}(\frac{\pi}{2}\textmd{Im} z^2)}{|z|^\beta}\right|^2 \left(\frac{1+|z^2|}{1+|\textmd{Im} z^2|}\right)^\alpha e^{-\pi|z|^2}\,dA(z) \nonumber\\
  & \asymp &  \int_{|z|>1}\frac{1}{|z|^{2\beta-2\alpha}(1+|\textmd{Im} z^2|)^\alpha}e^{\pi|\textmd{Im} z^2|-\pi|z|^2}\,dA(z) \nonumber\\
  & = & 8 \int_{1}^\infty\int_{0}^{x}\frac{e^{-\pi(x-y)^2}}{(x^2+y^2)^{\beta-\alpha}(1+xy)^\alpha}\,dxdy. \label{I-}
\end{eqnarray}}
It follows from a simple change of variables that
\begin{eqnarray}
J & := & \int_{0}^{x}\frac{e^{-\pi(x-y)^2}}{(x^2+y^2)^{\beta-\alpha}(1+xy)^\alpha}\,\,dy \nonumber\\
   & \asymp & \frac{1}{x^{2\beta-2\alpha}} \int_{0}^{\infty}\frac{e^{-\pi y^2}}{(1+x^2-xy)^\alpha}\chi_{[0,x]}(y)\, dy. \label{latter}
\end{eqnarray}
By \eqref{I-} and \eqref{latter}, we obtain
  \begin{eqnarray}
 I & = & \int_{1}^\infty\int_{0}^{\infty}\frac{1}{x^{2\beta-2\alpha}(1+x^2-xy)}\chi_{[0,x]}(y)e^{-\pi y^2}\,dxdy \nonumber \\
  & = & \int_{0}^\infty e^{-\pi y^2}dy\int_{\max\{1,y\}}^{\infty}\frac{dx}{x^{2\beta-2\alpha}(1+x^2-xy)^\alpha}dy. \nonumber
\end{eqnarray}
Consequently, the integral converges if and only if the following  integral,
$$\int_{\max\{1,y\}}^{\infty}\frac{dx}{x^{2\beta-2\alpha}(1+x^2-xy)^\alpha},$$
is also convergent, that is, if and only if $\beta > 1/2$. This completes the proof.
\end{proof}
The following lemma will be useful in the proofs of Theorems \ref{thm1ALS} and \ref{corGaborALS}.
\begin{lemma}\label{sep}
Let $\Lambda$ be a sequence satisfying the conditions of Theorem \ref{thm1ALS}. Then
\begin{eqnarray}\nonumber
\int_\mathbb{C} \left|F(z)\right|^2 \frac{\mathrm{dist}(z,\Gamma)^2}{\mathrm{dist}(z,\Lambda)^2} d\mu(z) \asymp \int_\mathbb{C} \left|F(z)\right|^2  d\mu(z)
\end{eqnarray}
for every entire function $F$ vanishing on $\Lambda$.
\end{lemma}
\begin{proof} Taking into account condition \ref{thm10ALS} in Theorem \ref{thm1ALS}, we can find a constant $c>0$ such that the disks $\left\{D(\lambda_n,c/\sqrt{n})\right\}$ are {disjoint}. We denote  $\mathcal{D}_\Lambda(c):=\cup D(\lambda_n,c/2\sqrt{n})$. Then
\begin{eqnarray}
\int_\mathbb{C} \left|F(z)\right|^2 \frac{\mathrm{dist}(z,\Gamma)^2}{\mathrm{dist}(z,\Lambda)^2} d\mu(z) & = & \int_{\mathbb{C}\setminus\mathcal{D}_\Lambda} \left|F(z)\right|^2 \frac{\mathrm{dist}(z,\Gamma)^2}{\mathrm{dist}(z,\Lambda)^2} d\mu(z) \nonumber\\
& & + \int_{\mathcal{D}_\Lambda} \left|F(z)\right|^2 \frac{\mathrm{dist}(z,\Gamma)^2}{\mathrm{dist}(z,\Lambda)^2} d\mu(z) \nonumber\\
     & \geq & \int_{\mathbb{C}\setminus\mathcal{D}_\Lambda} \left|F(z)\right|^2 d\mu(z). \label{ineq}
\end{eqnarray}
Furthermore, for every $n\geq 1$, we have
\begin{eqnarray}\nonumber
\int_{D(\lambda_n,c/2\sqrt{n})}\left|F(z)\right|^2d\mu(z) & \lesssim & \int_{D(\lambda_n,c/2\sqrt{n})} n\int_{C(\lambda_n,c/2\sqrt{n},c/\sqrt{n})}\left|F(w)\right|^2d\mu(w)dA(z) \\
    & \asymp & \int_{C(\lambda_n,c/2\sqrt{n},c/\sqrt{n})}\left|F(w)\right|^2d\mu(w). \nonumber
\end{eqnarray}
Summing over $n$ and using \eqref{ineq}, we obtain \begin{eqnarray}
\int_\mathbb{C} \left|F(z)\right|^2 \frac{\mathrm{dist}(z,\Gamma)^2}{\mathrm{dist}(z,\Lambda)^2} d\mu(z) & \gtrsim & \int_{\mathbb{C}} \left|F(z)\right|^2 d\mu(z). \nonumber
\end{eqnarray}
On the other hand,
{\small
\begin{eqnarray}
I_n & := &  \int_{D(\lambda_n,c/2\sqrt{n})}\left|F(z)\right|^2\frac{\mathrm{dist}(z,\Gamma)^2}{\mathrm{dist}(z,\Lambda)^2}d\mu(z) \nonumber\\
   & = & \int_{D(\lambda_n,c/2\sqrt{n})}\left|F(z)\frac{z-\gamma_m}{z-\lambda_n}\right|^2d\mu(z) \nonumber\\
   & \lesssim & \int_{D(\lambda_n,c/2\sqrt{n})} n\int_{C(\lambda_n,c/2\sqrt{n},c/\sqrt{n})}\left|F(w)\frac{w-\gamma_m}{w-\lambda_n}\right|^2d\mu(w)dA(z) \nonumber \\
    & \asymp & \int_{C(\lambda_n,c/2\sqrt{n},c/\sqrt{n})}\left|F(w)\right|^2d\mu(w). \nonumber
\end{eqnarray}}
Summing again over $n$ and using \eqref{ineq}, we get
\begin{eqnarray}
\int_\mathbb{C} \left|F(z)\right|^2 \frac{\mathrm{dist}(z,\Gamma)^2}{\mathrm{dist}(z,\Lambda)^2} d\mu(z) & \lesssim & \int_{\mathbb{C}} \left|F(z)\right|^2 d\mu(z). \nonumber
\end{eqnarray}
This completes the proof.
\end{proof}

\section{Proofs of the main results.}\label{coro-proof}
Theorems \ref{corA} and \ref{corAfull} will be obtained from Theorems \ref{thm1} and \ref{thm1full} respectively. The proof of Theorems \ref{thm2}, \ref{thm1}, \ref{thm2full} and \ref{thm1full} will be based on the estimates presented in the previous section and on the following result. 
\begin{lemma}\label{gen}
Let $\Lambda$ be a sequence of complex numbers. 
\begin{enumerate}[leftmargin=* ,parsep=0cm,itemsep=0cm,topsep=0cm]
\item Suppose that there exists an entire function $f$ such that $\mathrm{Z}(f) = \Lambda$ satisfying 
\begin{equation}
\left|f(z)\right|\ \lesssim\ e^{\frac{\pi}{2}|z|^2}\frac{1}{(1+|z|)^\alpha}\left(\frac{1+|z|}{1+|\textmd{Im} z|}\right)^{M},\quad z\in\mathbb{C}.
\end{equation}
If $\alpha>0$, then $\Lambda\setminus\{\lambda\}$, for some $\lambda\in\Lambda$ (for every $\lambda\in\Lambda$), is a zero set for $\mathcal{F}$. 
\item Assume that there exists an entire function $f$ that vanishes exactly in $\Lambda$ and verifies
\begin{equation}\label{2Lem}
e^{\frac{\pi}{2}|z|^2}\frac{\mathrm{dist}(z,\Lambda)}{(1+|z|)^\alpha}\left(\frac{1+|\textmd{Im} z|}{1+|z|}\right)^{M}\ \lesssim\ \left|f(z)\right|,\quad z\in\mathbb{C}.
\end{equation}
If $\alpha\leq 1$, then $\Lambda$ is a uniqueness set for $\mathcal{F}$. 
\end{enumerate}
\end{lemma}
\begin{proof}
\begin{enumerate}[leftmargin=* ,parsep=0cm,itemsep=0cm,topsep=0cm]
\item Let $\lambda\in\Lambda$ be fixed. The function $z\longmapsto\frac{f(z)}{z-\lambda}$ is holomorphic in $\mathbb{C}$ and 
{\small\begin{eqnarray}
\left\|\frac{f}{.-\lambda}\right\|^2\ \lesssim\ \int_\mathbb{C} \frac{1}{(1+|z|)^{2\alpha+2}}\left(\frac{1+|z|}{1+|\textmd{Im} z|}\right)^{2M}dm(z)\ \asymp\ \int_\mathbb{C} \frac{dm(z)}{(1+|z|)^{2\alpha+2}}. \nonumber
\end{eqnarray}}
The integral in the right hand side of the last inequality converges if and only if $\alpha>0$. 
\item Let $g\in\mathcal{F}$ be such that $\Lambda\subset\mathrm{Z}(g)$. There exists an entire function $h$ such that $g(z)=f(z)h(z)$, for every $z\in\mathbb{C}.$ It follows from the estimate in \eqref{2Lem} that 
\begin{eqnarray}
|h(z)|\frac{\mathrm{dist}(z,\Lambda)}{(1+|z|)^\alpha}\left(\frac{1+|\textmd{Im} z|}{1+|z|}\right)^{M}\ \lesssim\ \left|f(z)h(z)\right|e^{-\frac{\pi}{2}|z|^2}\ \lesssim\ 1,\quad z\in\mathbb{C}. \nonumber
\end{eqnarray}
Thus, the function $h$ is a polynomial of degree $k\leq \alpha.$ Integrating both sides of the last inequality we get
{\small\begin{eqnarray}
\int_\mathbb{C} \frac{\mathrm{dist}(z,\Lambda)^2}{(1+|z|)^{2(\alpha-k)}} \left(\frac{1+|\textmd{Im} z|}{1+|z|}\right)^{2M}\ dm(z) \lesssim\ \int_\mathbb{C} \left|f(z)h(z)\right|^2e^{-\pi |z|^2}\ dm(z)\ <\ \infty. \nonumber
\end{eqnarray}}
Since $\alpha\leq 1$, the integral in the left hand side of the latter inequality diverges, and, hence, $g$ is identically zero.
\end{enumerate}
\end{proof}

\begin{proof}[\bf Proof of Theorem \ref{thm1}]
Let $\Lambda$ be a sequence of complex numbers. Let $\varepsilon$ be a positive number such that $\delta=\delta(\Lambda)+\varepsilon<\nu$ and $\hat{\delta}=\hat{\delta}(\Lambda)-\varepsilon>\nu-1.$ Let $G_\Lambda$ be the infinite product associated to $\Lambda$. According to Lemma \ref{lem1}, the function $G_\Lambda$ satisfies the estimate 
\begin{eqnarray}
\frac{\mathrm{dist}(z,\Lambda)}{(1+|z|)^{\nu-\hat{\delta}}}\left(\frac{1+|\textmd{Im} z|}{1+|z|}\right)^M \ \lesssim\ \left|G_\Lambda(z)\right| e^{-\frac{\pi}{2}}\ \lesssim\ \frac{\mathrm{dist}(z,\Lambda)}{(1+|z|)^{\nu-\delta}}\left(\frac{1+|z|}{1+|\textmd{Im} z|}\right)^M .\nonumber
\end{eqnarray}
Since $\nu-1<\hat{\delta}\leq \delta<\nu$, Lemma \ref{gen} and the above estimates complete the proof.
\end{proof}

\begin{proof}[\bf Proof of Theorem \ref{thm2}]
By Lemma \ref{lem1}, we have
\begin{eqnarray}\label{KK}
\left|G_{\Lambda_{\beta,\nu}}(z)\right|\ \asymp\ \frac{e^{\frac{\pi}{2}|z|^2}}{(1+|z|)^{\nu+2\beta}}\mathrm{dist}(z,\Lambda_{\beta,\nu}),\quad \ z\in \mathbb{C}.
\end{eqnarray}
Lemma \ref{gen} and this estimate yield now the conclusions of Theorem \ref{thm2}. 
\end{proof}

\begin{proof}[\bf Proof of Theorem \ref{corA}]
In order to prove Theorem \ref{corA}, it suffices to check the conditions of Theorem \ref{thm1}. Indeed, by a simple computation, we obtain
\[|\lambda_{m+1,n}-\lambda_{m,n}|=|\lambda_{m+1,n}-\gamma_{m+1,n}+1+\gamma_{m,n}-\lambda_{m,n}|\geq 1-2\delta/N,\]
and
\[|\lambda_{m,n+1}-\lambda_{m,n}|=|\lambda_{m,n+1}-\gamma_{m,n+1}+i+\gamma_{m,n}-\lambda_{m,n}|\geq 1-2\delta/N.\]
Since $\delta<\frac{\min\{\nu,1-\nu\}}{2}$, the sequence $\Lambda$ is separated and hence condition \ref{thm10} of Theorem \ref{thm1} holds. We also have
\[\left|\gamma_{m,n}-\lambda_{m,n} \right|^2 =  |\gamma_{m,n}|^2\left[(1-e^{\delta_{m,n}})^2 + 4e^{\delta_{m,n}}\sin^2(\theta_{m,n}/2)\right], \]
where  $\lambda_{m,n}=\gamma_{m,n} e^{\delta_{m,n}}e^{i\theta_{m,n}}$. Using this equality and \eqref{cd}, we get
\[
\left|1-e^{\delta_{m,n}}\right|\lesssim 1/|\gamma_{m,n}|\ \ \mbox{and} \quad e^{\delta_{m,n}/2}\left|\sin(\theta_{m,n}/2)\right|\lesssim 1/|\gamma_{m,n}|,\quad n\geq 1.
\]
The sequences $\left(m\delta_{m,n}\right)$ and $\left(m\theta_{m,n}\right)$ are then bounded. Moreover,
{\small\begin{eqnarray}
\left|\underset{1\leq n\leq N}{\sum_{-s\leq k\leq s}} \delta_{k,n}\right|
   & \leq & \underset{1\leq n\leq N}{\sum_{-s\leq k\leq s}} \left|e^{\delta_{k,n}}-1\right|+O(1)\nonumber\\
   & \leq & \frac{\delta}{N} \underset{1\leq n\leq N}{\sum_{-s\leq k\leq s}}1/|\gamma_{k,n}| + O(1)\nonumber\\
   & = & 2\delta\log(s) + O(1). \nonumber
\end{eqnarray}}
It follows from this that 
\[
\limsup_{s\rightarrow\infty}\ \frac{1}{\log s}\ \left|\underset{1\leq n\leq N}{\sum_{-s\leq k\leq s}} \delta_{k,n}\right|\ \leq 2\delta.
\]
The condition \ref{thm12} of Theorem \ref{thm1} holds. The proof is now complete.
\end{proof}

\begin{proof}[\bf Proof of Theorem \ref{thm1full}]
The proof is the same as that of Theorem \ref{thm1}; we use Lemma \ref{lem1full} and Lemma \ref{gen}.
\end{proof}

\begin{proof}[\bf Proof of Theorem \ref{corAfull}]
To prove Theorem \ref{corAfull}, it is sufficient to verify that the conditions of Theorem \ref{thm1full} hold. Indeed, we have
\[|\lambda_{m+1,n}-\lambda_{m,n}|=|\lambda_{m+1,n}-\gamma_{m+1,n}+1+\gamma_{m,n}-\lambda_{m,n}|\geq 1-2\delta/|\gamma_{m,n}|
,\]
and
\[|\lambda_{m,n+1}-\lambda_{m,n}|=|\lambda_{m,n+1}-\gamma_{m,n+1}+i+\gamma_{m,n}-\lambda_{m,n}|\geq 1-2\delta/|\gamma_{m,n}|
.\]
The separation of the sequence $\Lambda$ follows from this and hence condition \ref{thm10full} of Theorem \ref{thm1full} holds. Condition \eqref{cdfull} and the equality
\[\left|\gamma_{m,n}-\lambda_{m,n} \right|^2 =  |\gamma_{m,n}|^2\left[(1-e^{\delta_{m,n}})^2 + 4e^{\delta_{m,n}}\sin^2(\theta_{m,n}/2)\right],
\]
where  $\lambda_{m,n}=\gamma_{m,n} e^{\delta_{m,n}}e^{i\theta_{m,n}}$, imply that
\[
\left|1-e^{\delta_{m,n}}\right|\lesssim 1/|\gamma_{m,n}|^2\ \ \mbox{and} \quad e^{\delta_{m,n}/2}\left|\sin(\theta_{m,n}/2)\right|\lesssim 1/|\gamma_{m,n}|^2,\quad n\geq 1.
\]
Consequently, $\left(\gamma_{m,n}^2\delta_{m,n}\right)$ and $\left(\gamma_{m,n}^2\theta_{m,n}\right)$ are bounded. Furthermore,
{\small\begin{eqnarray}
\left|\sum_{|\gamma_{m,n}|\leq R} \delta_{m,n}\right|
   & \leq & \sum_{|\gamma_{m,n}|\leq R} \left|e^{\delta_{m,n}}-1\right|+O(1)\nonumber\\
   & \leq & \delta \sum_{|\gamma_{m,n}|^2 \leq R}1/|\gamma_{m,n}|^2 + O(1)\nonumber\\
   & = & \delta\log(R) + O(1). \nonumber
\end{eqnarray}}
It follows that 
\[
\limsup_{R\rightarrow\infty}\ \frac{1}{\log R}\ \left|\sum_{|\gamma_{m,n}|^2 \leq R} \delta_{k,n}\right|\ \leq \delta.
\]
The condition \ref{thm12full} of Theorem \ref{thm1full} holds. The proof is now complete.
\end{proof}

\begin{proof}[\bf Proof of Theorem \ref{thm2full}]
By Lemma \ref{lem1full}, we have
{\small\begin{eqnarray}
\frac{\mathrm{dist}(z,\widetilde{\Lambda}_{\beta,\nu})}{(1+|z|)^{\nu+\beta}}\left(\frac{1+|\textmd{Im} z|}{1+|z|}\right)^M \ \lesssim\ \left|G_{\widetilde{\Lambda}_{\beta,\nu}}(z)\right| e^{-\frac{\pi}{2}|z|^2}\ \lesssim\ \frac{\mathrm{dist}(z,\widetilde{\Lambda}_{\beta,\nu})}{(1+|z|)^{\nu+\beta}}\left(\frac{1+|z|}{1+|\textmd{Im} z|}\right)^M .\nonumber
\end{eqnarray}}
Lemma \ref{gen} and this estimate imply  now the conclusions of Theorem \ref{thm2full}.
\end{proof}

In what follows we prove the results of  Subsection 1.3. The proofs of Theorems \ref{thm2ALS} and \ref{thm1ALS} are based on the estimates obtained in Lemmas \ref{lem1ALS} and \ref{gen}. We begin by the proof of Theorem \ref{thm1ALS}.

\begin{proof}[\bf Proof of Theorem \ref{thm1ALS}]
Let $\Lambda=\{\lambda_n,\ -\sqrt{2n},\ \pm i\sqrt{2n}\}\cup\{\pm 1\}$ be a sequence of $\mathbb{C}$ satisfying the conditions of Theorem \ref{thm1ALS}. 
%Since $\delta<1$
The infinite product $G_\Lambda$ satisfies the estimate
\begin{equation}
|G_\Lambda(z)|e^{-\frac{\pi}{2}|z|^2}\ \lesssim\ \frac{(1+|z|)^{2\delta+2M}\left|G_\Gamma(z)\right|}{(1+|\textmd{Im} z^2|)^M},\quad |\textmd{Re} z|\geq 1,\ |\textmd{Im} z|\geq 1. 
\end{equation}
By Lemma \ref{gen} and the fact that $\delta<1/4$, we get that $\Lambda\setminus\{\lambda\}$ is a zero set of $\mathcal{F}$. 

Now, let $\varepsilon$ be a small positive number such that $\delta=\delta(\Lambda)+\varepsilon<\frac{1}{4}.$ Assume that there exists a function $F\in\mathcal{F}\setminus\{0\}$ vanishing on $\Lambda$, and write $F(z)=h(z)G_\Lambda(z)$, for some entire function $h$. By Lemma \ref{lem1ALS}, we get
\begin{eqnarray}
\left|\frac{h(z) G_\Gamma(z)}{(1+|z|)^{2\delta}}\right|\ \lesssim\ \left|h(z)G_\Lambda(z)\right| = |F(z)|,\quad z\in\mathbb{C}\setminus\mathcal{R},\nonumber
\end{eqnarray}
and \begin{eqnarray}
\left|\frac{h(z) G_\Gamma(z)}{(1+|z|)^{2\delta+2M}}\right|\ \frac{\mathrm{dist}(z,\Lambda)}{\mathrm{dist}(z,\Gamma)} \lesssim\ \left|h(z)G_\Lambda(z)\right| = |F(z)|,\quad z\in\mathcal{R}\setminus\Gamma.\nonumber
\end{eqnarray}
Therefore, by Lemma \ref{sep}, we get
{\small\begin{equation}\label{prof}
\int_\mathcal{R}\left|\frac{h(z)G_\Gamma(z)}{(1+|z|)^{2\delta+2M}}\right|^2 d\mu(z)\ \lesssim\ \int_\mathcal{R}\left|F(z)\right|^2\frac{\mathrm{dist}(z,\Gamma)^2}{\mathrm{dist}(z,\Lambda)^2}d\mu(z)\lesssim \int_\mathbb{C}\left|F(z)\right|^2d\mu(z)< \infty,
\end{equation}}
and {\small\begin{equation}\label{prof11}
\int_{\mathbb{C}\setminus\mathcal{R}}\left|\frac{h(z) G_\Gamma(z)}{(1+|z|)^{2\delta+2M}}\right|^2d\mu(z)\leq \int_{\mathbb{C}\setminus\mathcal{R}}\left|\frac{h(z) G_\Gamma(z)}{(1+|z|)^{2\delta}}\right|^2d\mu(z)\ \lesssim\ \int_{\mathbb{C}\setminus\mathcal{R}} |F(z)|^2d\mu(z).
\end{equation}}
Now, fix $\gamma\in\Gamma$. We have $|z|^{2\delta}\ \lesssim\ |z-\gamma|^{1/2}\ \leq\  |z-\gamma|,\ \mbox{for every}\ |z|>2|\gamma|.$ Consequently,
$$\int_\mathbb{C} \left|\frac{h(z)G_\Gamma(z)}{(z-\gamma)P_{2M}(z)}\right|^2d\mu(z)\ <\ \infty,$$
for some polynomial $P_{2M}$ of degree $2M$ vanishing at $2M$ points of $\Gamma\setminus\{\gamma\}$. This means that  $z\mapsto\frac{h(z)G_\Gamma(z)}{(z-\gamma)P_{2M}} ~~\in \mathcal{F}$. On the other hand, since $\Gamma\setminus\{\gamma\}$ is a maximal zero sequence for $\mathcal{F}$, the function $h$ must be a polynomial of degree at most $2M$ (see \cite{zhu2011maximal,zhu2012analysis}). Hence, \eqref{prof11} becomes
\begin{equation}
\int_{\mathbb{C}\setminus\mathcal{R}}\left|\frac{G_\Gamma(z)}{(1+|z|)^{2\delta-k}}\right|^2d\mu(z)\asymp \int_{\mathbb{C}\setminus\mathcal{R}}\left|\frac{z^k G_\Gamma(z)}{(1+|z|)^{2\delta}}\right|^2d\mu(z) <\infty,
\end{equation}
for some integer $k\leq 2M$. Lemma \ref{7.3} ensures that $2\delta -k > 1/2$ and hence $h$ is constant. The identity \eqref{prof11} implies that the function $z\mapsto\frac{h(z)G_\Gamma(z)}{(1+|z|)^{2\delta}}$  belongs to $ L^2(\mathbb{C}\setminus\mathcal{R},d\mu)$, which is possible, by Lemma \ref{7.3}, only if $h=0$ (since $2\delta<1/2$). Finally, $\Lambda$ is a uniqueness set for $\mathcal{F}$. 
\end{proof}

\begin{proof}[\bf Proof of Theorem \ref{thm2ALS}]
By Lemma \ref{lem1ALS}, we have
\begin{eqnarray}
\left|G_{\Lambda_{\beta}}(z)\right|\ \asymp\ \frac{|G_\Gamma(z)|}{(1+|z|)^{2\beta}}\frac{\mathrm{dist}(z,\Lambda_{\beta})}{\mathrm{dist}(z,\Gamma)},\quad \ z\in \mathbb{C}\setminus\Gamma.\nonumber
\end{eqnarray}
Arguing as in the proof of Theorem \ref{thm1ALS} one can easily obtain the conclusions of Theorem \ref{thm2ALS}.
\end{proof}

\begin{proof}[\bf Proof of Theorem \ref{corAALS}]
The proof of Theorem \ref{corAALS} is similar to that of Theorems \ref{corA} and \ref{corAfull} and uses Theorem \ref{thm1ALS}. By a simple computation, we obtain
\[\sqrt{2n}\left|\lambda_{n+1}-\lambda_n\right|\geq (1-2\delta) + O(1/n).\]
Since $\delta<1/2$,  condition \ref{thm10ALS} of Theorem \ref{thm1ALS} holds. We also have
\[\left|\gamma_n-\lambda_n \right|^2 =  2n\left[(1-e^{\delta_n})^2 + 4e^{\delta_n}\sin^2(\theta_n/2)\right],\quad n\geq 1,\]
where  $\lambda_n=\sqrt{2n} e^{\delta_n}e^{i\theta_n}$. This equality together with \eqref{cdALS} ensures that 
the sequences  $\left(n\delta_n\right)$ and $\left(n\theta_n\right)$ are bounded. Moreover,
\begin{eqnarray}
\left|\sum_{k\leq n} \delta_k\right|
    \leq \sum_{k\leq n} \left|e^{\delta_k}-1\right|+O(1)
    \leq \frac{\delta}{2}\sum_{k\leq n}1/k + O(1)    = \frac{\delta}{2}\log(n) + O(1). \nonumber
\end{eqnarray}
Condition \ref{thm12ALS} of Theorem \ref{thm1ALS} follows immediately and the proof is complete.
\end{proof}

\begin{proof}[\bf Proof of Theorem \ref{corGaborALS}]
First, we consider the infinite product associated with $\Lambda$ defined as
\begin{align}\label{G}
G_\Lambda(z)=\lim_{r\rightarrow\infty} \prod_{\lambda\in\Lambda,\ |\lambda|\leq r}\left(1-\frac{z}{\lambda}\right),\quad z\in\mathbb{C}.
\end{align}
Using the same arguments as in the proof of Lemma \ref{lem1ALS}, we obtain
\begin{align}\label{4.2}
\frac{\mathrm{dist}(z,\Lambda)}{\mathrm{dist}(z,\Gamma)}\frac{(1+|\textmd{Im} z^2|)^M}{\left(1+|z|\right)^{2\delta+2M}}|G_\Gamma(z)| \lesssim \left|G_\Lambda(z)\right|,
\end{align}
and
\begin{align}\label{aa}
\left|G_\Lambda(z)\right|\lesssim\left|G_\Gamma(z)\right| \frac{(1+|z|)^{2\delta+2M}}{(1+|\textmd{Im} z^2|)^M}\frac{\mathrm{dist}(z,\Lambda)}{\mathrm{dist}(z,\Gamma)},
\end{align}
for every $z\in\mathbb{C}\setminus\Gamma$. Let $\gamma\in\Gamma$ be fixed and $P_{2M}$ be a polynomial of degree $2M$ vanishing at $2M$ points of $\Gamma\setminus\{\gamma\}$. By \eqref{4.2}, we have
\begin{equation}\label{dd}
\left|\frac{h(z)G_{\Gamma}(z)}{(z-\gamma)P_{2M}(z)|}\right|\lesssim \left|\frac{h(z)G_{\Gamma}(z)}{(1+|z|)^{2\delta}}\left(\frac{1+|\textmd{Im} z^2|}{1+|z|^{2}}\right)^M\right|\lesssim \left|F(z)\right|\frac{\mathrm{dist}(z,\Gamma)}{\mathrm{dist}(z,\Lambda)}.
\end{equation}
Next, we argue as in the proof of Theorem \ref{thm1} to obtain that $h$ is identically zero.\\

Now, we prove that $\Lambda\setminus\{\lambda\}$ is a zero set for $\mathcal{F}$, for fixed $\lambda\in\Lambda$. Indeed, identity \eqref{aa} implies that
\begin{equation}\label{4.5}
\frac{\mathrm{dist}(z,\Gamma)}{\mathrm{dist}(z,\Lambda)}\left|\frac{G_\Lambda(z)}{z-\lambda}\right| \lesssim \frac{\left|G_\Gamma(z)\right|}{(1+|z|)^{1-2\delta}}\left(\frac{1+|z|^2}{1+|\textmd{Im} z^2|}\right)^M,\quad |z|\geq 2|\lambda|,\ \ z\in\mathbb{C}\setminus\Lambda.
\end{equation}
Using this estimate and Lemmas \ref{7.3} and \ref{sep}, we obtain that $\frac{G_\Lambda}{z-\lambda}$ belongs to $\mathcal{F}$ and, hence, $\Lambda\setminus\{\lambda\}$ is a zero sequence for $\mathcal{F}$. This completes the proof.
\end{proof}

\section{Remarks}\label{rmks}

\subsection{Optimality of the condition on $\theta_\gamma$}
Here, we give two examples of sequences $(\theta_n)$ such that $(n\theta_n)$ goes (slowly) to infinity and for which Theorem \ref{thm1} and Theorem \ref{thm1ALS} do not hold. To do this, we need the following lemma. Set
\[ \mathcal{S} := \left\{z=|z|e^{i\theta}\ :\ \theta\in [\pi/16,\pi-\pi/16]\cup [\pi+\pi/16,2\pi-\pi/16] \right\}.\]

\begin{lemma}\label{lemoptimality}
Let $s\in(1/2,1)$ and consider the sequence $\theta_n=\pi/n^s$.  Let $\varphi$ be the meromorphic function defined as
\[\varphi(z)=\prod_{n\geq 1} \frac{1-z/n e^{i\theta_n}}{1-z/n} ,\quad z\notin\mathbb{N}.\]
Then, there exist $c_1,c_2>0$ such that 
\begin{equation}
e^{-c_1\sin(\theta)|z|^{1-s}}\ \lesssim\ \left|\varphi(z)\right|\ \lesssim\ e^{-c_2\sin(\theta)|z|^{1-s}},\quad z=|z|e^{i\theta}\in\mathcal{S}.
\end{equation}
\end{lemma}
\begin{proof}
Let $z=|z|e^{i\theta}\in\mathcal{S}$. We have
{\small\begin{align}
\log\varphi(z) & = \sum_{n\leq 2|z|} + \sum_{n> 2|z|} \log\left|\frac{1-z/n e^{2i\theta_n}}{1-z/n}\right| \nonumber\\
 & = \sum_{n\leq 2|z|} \log\left|1-\frac{n-ne^{i\theta_n}}{n-z}\right| +\sum_{n>2|z|} \log\left|1-\frac{1-e^{i\theta_n}}{1-z/n}\right| \nonumber\\
    &  = - \sum_{n\leq 2|z|} \textmd{Re}\left(\frac{n-ne^{i\theta_n}}{n-z}\right) + O\left(\frac{n^{2-2\nu}}{|z|^2}\right) -\sum_{n>2|z|} \textmd{Re}\left(\frac{1-e^{i\theta_n}}{1-z/n}\right)+ O\left(\theta_n^2\right) \nonumber \\
    &  = -|z|\sin(\theta)\sum_{n\leq 2|z|} \frac{\sin(\theta_n)}{n\left|1-z/n\right|^2}  - \sum_{n>2|z|} \frac{|z|\sin(\theta)}{n}\frac{\theta_n}{\left|1-z/n\right|^2} +O(1)\nonumber\\
    & \asymp -\sin(\theta)|z|^{1-s} + O(1). \label{*}
\end{align}}
\end{proof}
Now, we prove the following proposition
\begin{proposition}\label{prop1}
Let $s>0$ and let  $\Lambda =\left\{\pm \lambda_n,\ \pm i\sqrt{2n},\quad n\geq 1 \right\}\cup \{\pm 1\}$, where $\lambda_n:=\sqrt{2n}e^{i\theta_n}$ and $\theta_n=\pi/n^s$, for every $n\geq 1$. If $s\in(1/2,1)$, then $\Lambda$ is not a uniqueness set of zero excess for $\mathcal{F}.$
\end{proposition}

\begin{proof}
Let $G_\Lambda$ be the infinite product associated  to $\Lambda$, defined as
\[
G_\Lambda(z) = \prod_{n\geq 1}\left(1-\frac{z^2}{\lambda_n^2}\right)\left(1+\frac{z^2}{2n}\right).
\]
By the previous lemma we have
\begin{align}\label{gg}
\left|G_\Gamma(z)\right|e^{-c_1\sin(2\theta)|z|^{2-2\nu}} \lesssim \left|G_\Lambda(z)\right| \lesssim \left|G_\Gamma(z)\right|e^{-c_2\sin(2\theta)|z|^{2-2\nu}},\quad z=|z|e^{i\theta}\in \mathcal{S}.
\end{align}

Now, if for some $\lambda\in\Lambda$ the sequence  $\Lambda\setminus\{\lambda\}$ is a zero set for $\mathcal{F}$, then there exists a function $F\in\mathcal{F}$ vanishing exactly on $\Lambda\setminus\{\lambda\}$. By the Hadamard factorization theorem (see \cite{levin1996lectures}) and by the fact that every function of $\mathcal{F}$ is of order at most 2 and of finite type (see \cite{zhu2012analysis}), there exist two complex numbers $\alpha=\alpha_1+i\alpha_2$ and $\beta=\beta_1+i\beta_2$ such that $F(z)=\frac{G_\Lambda(z)}{z-\lambda}e^{\alpha z+\beta z^2}$, for all $z\in\mathbb{C}$. Hence,
\begin{eqnarray*}
\left|\frac{G_\Lambda(z)}{z-\lambda}e^{\alpha z+\beta z^2}\right| = O(e^{\frac{\pi}{2}|z|^2}),\quad z\in\mathbb{C}.
\end{eqnarray*}
This estimate and  \eqref{gg} imply that 
\begin{equation}
\left|\frac{G_\Gamma(z)}{1+|z|}e^{\alpha z+\beta z^2}\right|e^{-c_1\sin(2\theta)|z|^{2-2\nu}} = O(e^{\frac{\pi}{2}|z|^2}),\quad z\in\mathbb{C} ,\ |z|\geq 2|\lambda|.
\end{equation}
In particular, for $z=re^{i(\frac{\pi}{4}+k\frac{\pi}{2})}$, $k=0,1,2,3$, we get
{\small\begin{align*}
\frac{1}{1+r}\exp\left(\left[\alpha_1\cos(\frac{\pi}{4}+k\frac{\pi}{2}) -\alpha_2\sin(\frac{\pi}{4}+k\frac{\pi}{2})\right]r-(-1)^k\beta_2 r^2 -(-1)^kc_1 r^{2-2\nu}\right)
 = O(1)
\end{align*}}
for every $r>0\ \mbox{and}\ k=0,1,2,3$. Therefore, there exists $M >0$ such that
{\small$$\left\{\begin{array}{ccc}
-\beta_2 r^2 + \frac{\sqrt{2}}{2}(\alpha_1-\alpha_2) r - c_1 r^{2-2\nu} \leq M + \log(1+r), \\[0.1cm]
\beta_2 r^2 - \frac{\sqrt{2}}{2}(\alpha_1+\alpha_2) r + c_1 r^{2-2\nu} \leq M + \log(1+r), \\[0.1cm]
-\beta_2 r^2 - \frac{\sqrt{2}}{2}(\alpha_1-\alpha_2) r - c_1 r^{2-2\nu} \leq M + \log(1+r), \\[0.1cm]
\beta_2 r^2 + \frac{\sqrt{2}}{2}(\alpha_1+\alpha_2) r + c_1 r^{2-2\nu} \leq M + \log(1+r).
\end{array}
\right.$$}
These inequalities imply that $\alpha_1=0$, $\alpha_2=0$, $\beta_2=0$, and we obtain finally the  inequalities $c_1r^{2-2\nu} \leq M + \log(1+r)$ and $-c_1r^{2-2\nu} \leq M + \log(1+r)$ which cannot hold together since $c_1\neq 0$. Consequently, for every $\lambda\in\Lambda$, the sequence $\Lambda\setminus\{\lambda\}$ is not a zero set for $\mathcal{F}$ and so $\Lambda$ is not a uniqueness set of zero excess for $\mathcal{F}$.
\end{proof}

\begin{proposition}
Let $s>0$, $N\geq 1$. Set
  $$\Lambda = \{\gamma_{m,n}\in\Gamma_\nu\ :\ (m,n)\in\mathbb{Z}\times\mathbb{Z}\setminus\{1\}\}\cup \left\{(m+i)e^{i\theta_m}\ :\ m\in\mathbb{Z} \right\},$$
where $\theta_m=\pi/m^s$, for every $m\geq 1$ and $\theta_m=0$, for every $m\leq 0$. If $s\in(1/2,1)$, then $\Lambda$ is not a uniqueness set of zero excess for $\mathcal{F}.$
\end{proposition}
\begin{proof}
Let $G_\Lambda$ be the infinite product associated with $\Lambda$, as in Lemma \ref{lem1}. We first write
\[
\frac{G_\Lambda(z)}{G_{\Gamma_\nu}(z)} = \prod_{m\geq 1} \frac{1-z/(m+i)e^{\theta_m}}{1-z/m}.
\]
Recall first that $(2)$ in Lemma \ref{lem1} implies that $$\left|G_{\Gamma_{\nu}}(z)\right|\asymp e^{\frac{\pi}{2}|z|^2}\frac{\mathrm{dist}(z,\Gamma_{\nu})}{(1+|z|)^\nu}, \quad z\in\mathbb{C}.$$
{By Lemma \ref{lemoptimality} and the estimate of the function $G_{\Gamma_\nu}$, $G_\Lambda$ satisfies the estimates}
\begin{equation}
\frac{e^{\frac{\pi}{2}|z|^2}}{1+|z|^\nu}e^{-c_1\sin(\theta)|z|^{1-s}}\ \lesssim\ \left|G_\Lambda(z)\right|\ \lesssim\ \frac{e^{\frac{\pi}{2}|z|^2}}{1+|z|^\nu}e^{-c_2\sin(\theta)|z|^{1-s}},\quad z=|z|e^{i\theta}\in\mathcal{S}.
\end{equation}
Now, if $f$ is a function of $\mathcal{F}$ that vanishes on $\Lambda\setminus\{\lambda\}$, for some fixed $\lambda\in\Lambda$, then it can be written as $f=h \frac{G_\Lambda}{z-\lambda}$, for some entire function $h$. Since
\[
|h(z)|\frac{e^{\frac{\pi}{2}|z|^2}}{(1+|z|)^{1+\nu}}e^{-c_1\sin(\theta)|z|^{1-s}}\ \lesssim\ \left|h(z)\frac{G_\Lambda(z)}{z-\lambda}\right|\ \lesssim\ e^{\frac{\pi}{2}|z|^2},\quad z=|z|e^{i\theta}\in\mathcal{S},
\]
the function $h$ is a constant and the latter inequality becomes
\[
\frac{1}{(1+|z|)^{1+\nu}}e^{-c_1\sin(\theta)|z|^{1-s}}\ \lesssim\  1,\quad z=|z|e^{i\theta}\in\mathcal{S},
\]
which is not possible. This completes the proof of the proposition.
\end{proof}

\subsection{Relation between $\delta(\Lambda)$ and condition \eqref{(b)}}
\begin{lemma}\label{equiv}
Let $\left(\delta_n\right)_{n\geq 1}$ be a sequence of real numbers such that $\left(n\delta_n\right)_{n\geq 1}$ is bounded and let $\delta$ a positive real number.
Assume that there exists an integer $N\geq 1$ such that
\begin{eqnarray}\label{avdd}
\underset{n\geq 0}{\sup}\ \frac{n+1}{N} \left|\underset{k=n+1}{\overset{n+N}{\sum}}\delta_k\right|  < \delta.
\end{eqnarray}
Then {\small$\underset{n\rightarrow\infty}{\limsup}\,\ \frac{1}{\log n}\left|\underset{k\leq n}{\sum}\ \delta_{k}\right| < \delta.$ }The converse is not true.
\end{lemma}

\begin{proof}
Suppose that \eqref{avdd} holds for some integer $N\geq 1$. Let $m$ be a large integer. There exists $(p,r)\in\mathbb{N}\times\{0,\ldots,N-1\}$ such that $m=pN+r$ and hence
{\small
\begin{align}
\left|\underset{k=1}{\overset{m}{\sum}}\ \delta_k\right|& \leq  \underset{j=0}{\overset{p-1}{\sum}}\frac{N}{jN+1}\frac{jN+1}{N}\left|\underset{k=jN+1}{\overset{(j+1)N}{\sum}}\delta_k\right|+\left|\underset{k=pN+1}{\overset{pN+r}{\sum}}\delta_k\right|\nonumber\\
& \leq \underset{j=0}{\overset{p-1}{\sum}}\frac{N}{jN+1}\ \underset{n\geq 0}{\sup}\ \frac{n+1}{N}\left|\underset{k=n+1}{\overset{n+N}{\sum}}\delta_k\right| +O(1) \nonumber\\
& =   \underset{n\geq 0}{\sup}\ \frac{n+1}{N} \left|\underset{k=n+1}{\overset{n+N}{\sum}}\ \delta_k\right|\ \log m + O(1).    \nonumber
\end{align}}
Thus,
\begin{eqnarray}\nonumber
\underset{m\rightarrow\infty}{\limsup}\ \frac{1}{\log m}\left|\underset{k=1}{\overset{m}{\sum}}\ \delta_k\right| & \leq & \underset{n\geq 0}{\sup}\ \frac{n+1}{N} \left|\underset{k=n+1}{\overset{n+N}{\sum}}\ \delta_k\right|.
\end{eqnarray}
To prove that the converse is not true, we consider the sequence $\left(\delta_n\right)$ defined as
\[\delta_n = \frac{(-1)^k}{2^k},\quad \mbox{if}\ \ 2^k\leq n< 2^{k+1}.\]
Let $N$ be a fixed integer.  By a simple calculation, we obtain
\begin{eqnarray}
\underset{m\rightarrow\infty}{\limsup}\ \frac{1}{\log m} \left|\sum_{k\leq m}\delta_k \right| = 0 < 1 \leq   \underset{n\geq 0}{\sup}\ \ \frac{n+1}{N}\left|\sum_{k=n+1}^{n+N} \delta_k \right|. \nonumber
\end{eqnarray}
This completes the proof of Lemma \ref{equiv}.
\end{proof}

\subsection*{Acknowledgments}
The author is deeply grateful to Alexander Borichev, Omar El-Fallah and Karim Kellay for their helpful discussions, suggestions and remarks. The author is also thankful to the referees for their helpful remarks and suggestions. This research is supported by CNRST Grant 84UM52016

\end{document}